\documentclass[11pt, reqno]{amsart}
\usepackage{hyperref}
\hypersetup{
	colorlinks=true,
	linkcolor=blue,
	filecolor=blue,
	urlcolor=red,
	citecolor=green,}

\usepackage{CJK}
\usepackage{caption}
\newtheorem{thm}{Theorem}[section]
\newtheorem{cor}[thm]{Corollary}
\newtheorem{lem}[thm]{Lemma}
\newtheorem{prop}[thm]{Proposition}

\newtheorem{rem}{Remark}
\newtheorem*{thm*}{Theorem}

\theoremstyle{definition}
\newtheorem{defn}[thm]{Definition}
\theoremstyle{remark}
\usepackage{tikz}
\usetikzlibrary {patterns}

\numberwithin{equation}{section}

\usepackage{subfigure}
\usepackage{bm}

\newcommand{\thmref}[1]{Theorem~\ref{#1}}

\newcommand{\lemref}[1]{Lemma~\ref{#1}}

\newcommand{\propref}[1]{Proposition~\ref{#1}}
\newcommand{\corref}[1]{Corollory~\ref{#1}}

\newcommand{\bR}{{\mathbb R}}

\newcommand{\mL}{{\mathcal L}}

\newcommand{\ric}{{\rm Ric}}

\newcommand{\di}{{\rm div}}

\newcommand{\la}{\bm{\langle}}
\newcommand{\ra}{\bm{\rangle}}

\newcommand{\vol}{{\rm{vol}}}

\begin{document}
	
\title[On the universal local and global properties of positive solutions]{On the universal local and global properties of positive solutions to $\Delta_pv+b|\nabla v|^q+cv^r=0$ on complete Riemannian manifolds}

\author{Jie He}
\address{School of Mathematics and Physics, Beijing University of Chemical Technology,  Chaoyang District, Beijing 100029, China}
\email{hejie@amss.ac.cn}
\author{Youde Wang}
\address{1. School of Mathematics and Information Sciences, Guangzhou University; 2. Hua Loo-Keng Key Laboratory
		Mathematics, Institute of Mathematics, Academy of Mathematics and Systems Science, Chinese Academy
		of Sciences, Beijing 100190, China; 3. School of Mathematical Sciences, University of Chinese Academy of Sciences,
		Beijing 100049, China.}
	\email{wyd@math.ac.cn}
	
	
	
	
	
\begin{abstract}
In this paper we study the positive solutions to a nonlinear elliptic equation $$\Delta_pv+b|\nabla v|^q+cv^r =0$$ defined on a complete Riemannian manifold $(M,g)$ with Ricci curvature bounded from below, where $p>1$, $q,\, r, \, b$ and $c$ are some real constants. If $p>1$ is given and $bc\geq 0$, we provide a new routine to give some regions of $(q, r)$ such that the Cheng-Yau's logarithmic gradient estimates hold true exactly on such given regions. In particular, we derive the upper bounds of the constants $c(n, p, q, r)$ in the Cheng-Yau's gradient estimates for the entire solutions to the above equation. As applications, we reveal some universal local and global properties of positive solutions to the equation. On the other hand, we extend some results due to \cite{MR1879326} to the case the domain of the equation is a complete manifold and obtain wider ranges of $(q,r)$ for Liouville properties.
\end{abstract}
\maketitle

\textbf{\tableofcontents}
	
\section{Introduction}
Liouville property states that entire solutions to an elliptic equation must be constants under certain boundedness conditions. The earliest Liouville theorem can be traced back to Cauchy. In 1844, Cauchy proved that any bounded entire holomorphic function must be a constant. The Cauchy's theorem was extended at once to real harmonic functions on $\bR^n$ which are bounded only on one side.
The Liouville theorem has had a deep impact across many fields, such as complex analysis, PDE and differential geometry (see \cite{CM}).

In the past half century, one trend about the study of Liouville theorem is to study the Liouville property of semilinear and quasi-linear elliptic equations defined on Riemannian manifolds (see for example\cite{MR385749, MR1491451, MR1472893, MR2518892, Li-Z, MR2880214, MR4559367, MR431040}). Up to now, it is well-known that the curvature of a Riemannian manifold $(M, g)$ could affects some global properties of solution to partial differential equations on $M$ and global quantities like the eigenvalues of the Laplacian with respect to $g$ on $M$. So, the situation is more complicated than the case $M\equiv\mathbb{R}^n$. For instance, the so-called Cheng-Yau type gradient estimate, a more profound property of solutions to elliptic equations on manifolds, is of obvious geometric significance and plays important role on the local and global properties of solutions to a linear or nonlinear elliptic equation on a complete manifold.

Let's recall that Cheng and Yau \cite{MR385749} in 1975 used Bochner’s formula on Riemannian manifolds and employed cut-off functions via the Laplacian comparison theorem under the condition of Ricci curvature being bounded below to prove that any positive harmonic function $v$ defined on a complete non-compact Riemannian manifold with $\mathrm{Ric}_g\geq-(n-1)\kappa g$ satisfies
\begin{align}\label{yaus}
\sup_{B_{R/2}(o)}\frac{|\nabla v|}{v}\leq C_n\frac{1+\sqrt\kappa R}{R}.
\end{align}
The above elegant logarithmic estimate is also holds for the eigenvalue problem (see \cite{Li-Yau}). Now, such an estimate is called as Cheng-Yau's gradient estimate. An important feature of Cheng-Yau's estimate is that the RHS of estimate \eqref{yaus} depends only on $n$, $\kappa$ and $R$, it does not depend on the injective radius or global system. If $(M,g)$ is a non-compact complete manifold with non-negative Ricci curvature, then Liouville theorem for harmonic functions follows immediately from Cheng-Yau's gradient estimate. Moreover, Harnack inequality can also be concluded from the estimate \eqref{yaus}.

Usually, the above estimate is also called as universal a priori estimates. By using the word ``universal" here, we mean that our bounds are not only independent of any given solution under consideration but also do not require, or assume, any boundary conditions whatsoever. Actually, to derive universal a priori estimates for solutions of some partial differential equations is also one's principal purpose to explore local and global properties of these solutions (see \cite{MR615628, Dancer, {MR2350853}, MR1946918}).

Let $(M,g)$ be a complete non-compact Riemannian manifold with $\mathrm{Ric}_g\geq-(n-1)\kappa g$. We say an elliptic equation defined on $M$ satisfies Cheng-Yau type gradient estimate if for any geodesic ball $B_R(o)\subset M$ and any (positive) solution $v$ of the equation on $B_R(o)$, there holds the above estimate \eqref{yaus}.

{\em A natural question is for which equations on a complete non-compact Riemannian manifold with $\mathrm{Ric}_g\geq-(n-1)\kappa g$ the Cheng-Yau's estimate holds true exactly?}

Indeed, one has shown that for some equations, for instances $p$-harmonic functions, Lane-Emden equation etc., such a priori estimates hold true (see \cite{{MR2880214}, {MR4559367}, {he2023gradient}, {he2023nashmoser}}).

In this paper, we intend to using directly Nash-Moser iteration to establish such a gradient estimate and Liouville properties of positive solutions to the following equation:
\begin{align}\label{equ0}
\begin{cases}
\Delta_pv+b|\nabla v|^q+cv^r=0; \\
p>1,\,\, b,\, c, \, q,\, r \in \bR,
\end{cases}
\end{align}
on a complete Riemannian manifold $(M,g)$ with $\mathrm{Ric}_g\geq-(n-1)\kappa g$. It is worthy to point out that, in order to obtain Harnack inequalities and universal boundedness estimates, although one usually needs to employ the classical Nash-Moser iteration, as adopted by Gidas-Spruck \cite{MR615628} or Serrin-Zou \cite{MR1946918}, or blow-up analysis method via Liouville theorems inspired by Dancer \cite{Dancer} or Pol\'a\v{c}cik-Quittner-Souplet \cite{MR2350853}, neither of these methods can be directly applied to the above equations \eqref{equ0} on a complete Riemannian manifold. So, here we will adopt a new way to apply Nash-Moser technique to approach these estimates and provide a new routine to explore the deep relationship between these estimates and Riemannian curvature.

When equation \eqref{equ0} is defined on $\mathbb R^n$, it has been studied by many mathematicians, for details we refer to \cite{MR4150912, MR4043662, filippucci2022priori,MR1879326} and references therein. Obviously, the equation \eqref{equ0} is related to many classical equations. For instance, in the case $b=c=0$, equation \eqref{equ0} reduces to the $p$-Laplace harmonic function equation. In 2011, Wang and Zhang \cite{MR2880214} proved any positive $p$-harmonic function $v$ on  $(M,g)$ with $\mathrm{Ric}_g\geq-(n-1)\kappa g$ satisfies
\begin{align*}
\sup_{B_{R/2}(o)}\frac{|\nabla v|}{v}\leq C_{n,p}\frac{1+\sqrt\kappa R}{R}.
\end{align*}
Wang-Zhang's results generalized Cheng-Yau's result \cite{MR431040} of $p=2$ to any $p>1$ and improved Kotschwar-Ni's results \cite{MR2518892} by weakening sectional curvature condition to the Ricci curvature condition. Later  in 2014, Sung and Wang \cite{MR3275651} studied the sharp gradient estimate for eigenfunctions of $p$-Laplace operator.
	
In the case $b=-1$ and $c=0$, then equation (\ref{equ0}) reduces to the so called quasi-linear Hamilton-Jacobi equation,
\begin{align}\label{hamjoco}
\Delta_p v - |\nabla v|^q = 0.
\end{align}
For the case $(M, g)$ is the Euclidean space $\bR^n$, we refer to \cite{MR3261111, MR833413} for some Liouville type results of global solutions of equation \eqref{hamjoco}. In particular, Bidaut-V\'eron, Garcia-Huidobro and V\'eron \cite{MR3261111} proved that, if $n\geq p>1$ and $q>p-1$, then any solution $u$ of equation  (\ref{hamjoco}) defined on a domain $\Omega\subset\bR^n$ satisfies
	\begin{align}\label{jfathma}
		|\nabla v(x)|\leq c(n,p,q)\left(d(x,\partial\Omega)\right)^{-\frac{1}{q+1-p}},\quad \forall x\in\Omega\subset\mathbb{R}^n.
	\end{align}
In \cite{MR3261111}, the authors also extended their estimates to equations on complete non-compact manifolds with sectional curvature bounded from below. Recently, Han, He and Wang \cite{han2023gradient} also studied the nonlinear Hamilton-Jacobi equation on manifolds $(M,g)$ with $\mathrm{Ric}_g\geq-(n-1)\kappa g$. They gave a unified Cheng-Yau type gradietn estimate for the solutions of equation \eqref{hamjoco} with $q>p-1$ defined on $M$. Concretely, for any solution $u$ to \eqref{hamjoco} there holds the following universal estimate
	$$
	\sup_{B_{\frac{R}{2}}(o)}|\nabla v|\leq C_{n,p,q}\left( \frac{1+\sqrt{\kappa}R  }{R}\right)^{\frac{1}{q-p+1}}.
	$$
	
If $b=0$ and $c=1$, then equation \eqref{equ0} reduces to the Lane-Emden equation
\begin{align}\label{equa:1.3}
	\begin{cases}
		\Delta_pv + cv^r=0, \quad &\text{in}\,  M;\\
		v>0,\quad &\text{in}\, M.
	\end{cases}
\end{align}
In the case $M=\bR^n$, $r>0$ and $1<p<n$, there have been extensive studies on the equation \eqref{equa:1.3} over the recent decades; see \cite{MR1004713, MR1134481, MR982351, MR1121147, MR615628, MR829846, MR2350853, MR1946918, MR2522424} and references therein. Here we recall only some previous work on the gradient estimates of positive solutions to Lane-Emden equation on a complete manifold $(M,g)$ with $\mathrm{Ric}_g\geq -(n-1)\kappa$.

In the case $p=2$ and $c>0$, Wang-Wei \cite{MR4559367} derived Cheng-Yau type gradient estimates for positive solutions to \eqref{equa:1.3} under the assumption
	$$
	r\in \left(-\infty,\quad \frac{n+1}{n-1}+\frac{2}{\sqrt{n(n-1)}}\right).
	$$
For any $p>1$ and $c>0$, He-Wang-Wei \cite{he2023gradient} proved that the Cheng-Yau type gradient estimate holds for positive solutions of \eqref{equa:1.3} when
$$
r\in \left(-\infty,\quad \frac{n+3}{n-1}(p-1)\right).
$$
We refer to \cite{MR3912761} when the $p$-Laplace operator in the equation \eqref{equa:1.3} is replaced by weighted $p$-Laplace Lichnerowicz operator.
	
Now we turn to the case where $b$ and $c$ are both non-vanishing. The following theorem was proved by Bidaut-V\'eron, Garcia-Huidobro and V\'eron(\cite{MR4150912}) for $p=2$ in 2020 and by Filippucci,  Sun and Zheng \cite{filippucci2022priori} for any $p>1$ in 2022.

\begin{thm*} (Filippucci, Sun and Zheng \cite{filippucci2022priori}) Let $v$ be a positive solution to the following equation
\begin{align} \label{sun1}
\begin{cases}
\Delta_p v + M|\nabla v|^q + v^r=0, \quad\mbox{in}\,\, \Omega\subset\bR^n \\
p>1,\,\, M>0.
\end{cases}
\end{align}
If $(q, r)$ satisfies
\begin{align}\label{sun}
1<\frac{q}{p-1}<\frac{n+2}{n}\quad\text{and}\quad 1<\frac{r}{p-1}<\frac{n+3}{n-1},
\end{align}
then there exist positive constants $a$ and $c(n,p,q,r)$ such that
\begin{align}\label{sun2}
|\nabla v^a(x)|\leq c(n,p, q,r)\left(d(x,\partial\Omega) \right)^{-\frac{pa}{r-p+1}-1}, \quad\forall x\in\Omega.
\end{align}
\end{thm*}
Naturally, one wants to know if the similar gradient estimate with that established in the above theorem for the case $\Omega\subset\bR^n$ also holds true for the same equation defined on a complete Riemannian manifold?

Furthermore, one would like to ask whether or not any positive solution to equation (\ref{equ0}) on a complete non-compact Riemannian manifold with Ricci curvature bounded from below satisfies the exact Cheng-Yau type gradient estimates?

It turns out that the answer is yes. What's more, the Cheng-Yau type estimate can be established in a wider range of $(q, r)$ than the range in \eqref{sun}.
Now we are in a position to state our results.

\subsection{Local gradient estimate}\

The first goal of this article is to prove the following Cheng-Yau's local estimates holds true exactly:
\begin{thm}\label{t1}
Let $(M,g)$ be an $n$-dim ($n\geq 2$) complete Riemannian manifold satisfying $\mathrm{Ric}_g\geq-(n-1)\kappa g$ for some constant $\kappa\geq0$. Let  $v\in C^1(B_R(o)) $ be a positive solution  of equation (\ref{equ0}) on a geodesic ball $B_R(o)$. We define the following four regions,
	\begin{align*}
	W_1=&\left\{(b, c,  q, r): b\left(\frac{n+1}{n-1}-\frac{q}{p-1}\right)\geq 0 \quad \text{and}\quad
	c\left(\frac{n+1}{n-1}-\frac{r}{p-1}\right)\geq 0 \right\}	 ;
	\\
	W_2=&\left\{(b, c,  q, r): b\left(\frac{n+1}{n-1}-\frac{q}{p-1}\right)\geq 0 \quad \text{and}\quad
	\left|\frac{n+1}{n-1}-\frac{r}{p-1}\right|<\frac{2}{n-1} \right\};
	\\
	W_3=&\left\{(b, c, q, r):
	\left|\frac{n+1}{n-1}-\frac{q}{p-1}\right|<\frac{2}{n-1}  \quad \text{and}\quad
	c\left(\frac{n+1}{n-1}-\frac{r}{p-1}\right)\geq 0 \right\};
	\\
	W_4=&\left\{(b, c, q, r):
	\left( \frac{n+1}{n-1}-\frac{q}{p-1}\right)^2
	+\left(\frac{n+1}{n-1}-\frac{r}{p-1}\right)^2
	<\frac{4}{(n-1)^2}\right\}.
\end{align*}
If $p>1$, $bc\geq 0$ and $(b, c, q, r)\in W_1\cup W_2\cup W_3\cup W_4$, then there holds true
\begin{align}\label{gradient}
\sup_{B_{\frac{R}{2}}(o)}\frac{|\nabla v|}{v}\leq C(n,p,q,r) \frac{1+\sqrt{\kappa}R}{R} ,
\end{align}
where the constant $C(n,p,q,r)$ depends only on $n,\,p,\,q$ and $r$, and may vary in different $W_i$ where $i=1,\,2,\,3,\,4$ (see Figure 1).
\end{thm}

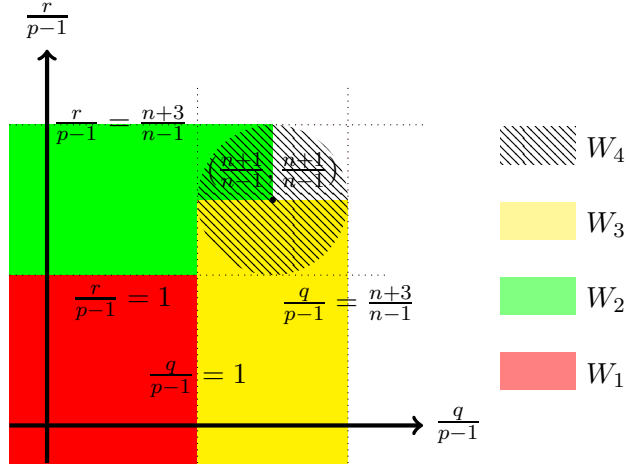
\begin{figure}
\begin{tikzpicture}
		\pgfsetfillopacity{0.5}
		\draw[help lines, color=red!5, dashed] (-0.5,-0.5) grid (4,4);
		\path[fill=red](-0.5, 3)--(3,3)--(3,-0.5)--(-0.5, -0.5);
		\path[fill=green](-0.5, 4)--(3,4)--(3,2)--(-0.5, 2);
		\path[fill=yellow](2, -0.5)--(2,3)--(4,3)--(4, -0.5);
		\fill[pattern=north west lines](3,3) circle (1);
		\draw[dotted] (4 ,-0.5) --(4, 2)  node[anchor=north]{$\frac{q}{p-1}= \frac{n+3}{n-1}$} -- (4 ,4) -- (4 ,4.5);
		\draw[dotted] (2,-0.5) --(2, 1)  node[anchor=north]{$\frac{q}{p-1}=1 $} -- (2,4) -- (2,4.5);
		\draw[dotted] (-0.5, 4 )--(2, 4) node[anchor=east]{$\frac{r}{p-1}= \frac{n+3}{n-1}$}-- (4, 4 )  -- (5, 4);
		\draw[dotted] (-0.5,2) --(1,2)  node[anchor=north]{$\frac{r}{p-1}= 1$} -- (4,2) -- (4.5,2);
		\filldraw[black] (3,3) circle (1pt) node[anchor=south]{$(\frac{n+1}{n-1}, \frac{n+1}{n-1})$};
		\draw[->,ultra thick] (-0.5,0)--(5,0) node[right]{$\frac{q}{p-1}$};
		\draw[->,ultra thick] (0,-0.5)--(0,5) node[above]{$\frac{r}{p-1}$};
\end{tikzpicture}			
\begin{tikzpicture}
		\path[fill=green!0](2, 2)--(3,2)--(3,2.5)--(2,2.5);
		\path[fill=red!50](2, 3)--(3,3)--(3,3.2) node[anchor=west]{$W_1$}--(3,3.5)--(2,3.5);				
		\path[fill=green!50](2, 4)--(3,4)--(3,4.2) node[anchor=west]{$W_2$}--(3,4.5)--(2,4.5);
		\path[fill=yellow!50](2, 5)--(3,5)--(3,5.2) node[anchor=west]{$W_3$}--(3,5.5)--(2,5.5);
		\path[pattern=north west lines](2, 6)--(3,6)--(3,6.2) node[anchor=west]{$W_4$}--(3,6.5)--(2,6.5);
\end{tikzpicture}
\caption{When $b>0, c>0$, the region of $(\frac{q}{p-1}, \frac{r}{p-1})$ in $W_1, W_2, W_3, W_4$.}
\end{figure}

	\begin{figure}	
	\begin{tikzpicture}
		\pgfsetfillopacity{0.5}
		\draw[help lines, color=red!5, dashed] (-0.5,-0.5) grid (4,4);
		\path[fill=red](7, 3)--(3,3)--(3,6)--(7, 6);
		\path[fill=green](7, 4)--(3,4)--(3,2)--(7, 2);
		\path[fill=yellow](2, 6)--(2,3)--(4,3)--(4, 6);
		\fill[pattern=north west lines](3,3) circle (1);
		\draw[dotted] (4 ,-0.5) --(4, 2)  node[anchor=north]{$\frac{q}{p-1}= \frac{n+3}{n-1}$} -- (4 ,4) -- (4 ,6);
		\draw[dotted] (2,-0.5) --(2, 1)  node[anchor=north]{$\frac{q}{p-1}=1 $} -- (2,4) -- (2,6);
		\draw[dotted] (-0.5, 4 )--(2, 4) node[anchor=east]{$\frac{r}{p-1}= \frac{n+3}{n-1}$}-- (4, 4 )  -- (7, 4);
		\draw[dotted] (-0.5,2) --(1,2)  node[anchor=north]{$\frac{r}{p-1}= 1$} -- (4,2) -- (7,2);
		\filldraw[black] (3,3) circle (1pt) node[anchor=south]{$(\frac{n+1}{n-1}, \frac{n+1}{n-1})$};
		\draw[->,ultra thick] (-0.5,0)--(7,0) node[right]{$\frac{q}{p-1}$};
		\draw[->,ultra thick] (0,-0.5)--(0,6) node[above]{$\frac{r}{p-1}$};
	\end{tikzpicture}		
	\begin{tikzpicture}
		\path[fill=green!0](2, 2)--(3,2)--(3,2.5)--(2,2.5);
		\path[fill=red!50](2, 3)--(3,3)--(3,3.2) node[anchor=west]{$W_1$}--(3,3.5)--(2,3.5);				
		\path[fill=green!50](2, 4)--(3,4)--(3,4.2) node[anchor=west]{$W_2$}--(3,4.5)--(2,4.5);
		\path[fill=yellow!50](2, 5)--(3,5)--(3,5.2) node[anchor=west]{$W_3$}--(3,5.5)--(2,5.5);
		\path[pattern=north west lines](2, 6)--(3,6)--(3,6.2) node[anchor=west]{$W_4$}--(3,6.5)--(2,6.5);
	\end{tikzpicture}	
	\caption{When $b<0, c<0$, the region of $(\frac{q}{p-1}, \frac{r}{p-1})$ in $W_1, W_2, W_3, W_4$.}
\end{figure}
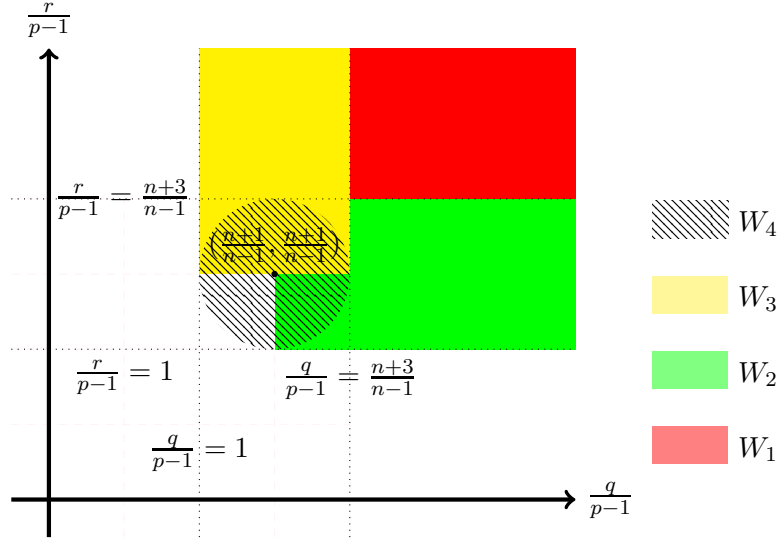

In \thmref{t1} the only assumption on $b$ and $c$ is $bc\geq 0$, i.e., the signs of $b$ and $c$. If we divide the assumption on $b$ and $c$ into two cases: (i) $b\geq0 $ and $c\geq 0$; and (ii) $b\leq 0$ and $c\leq 0$, then it is easy to see that the following conclusions hold.

\begin{cor}
Let $(M,g)$ be an $n$-dim ($n\geq 2$) complete Riemannian manifold satisfying $\mathrm{Ric}_g\geq-(n-1)\kappa g$ for some constant $\kappa\geq0$. Assume that $p>1$ and $v\in C^1(B_R(o))$ be a positive solution of equation (\ref{equ0}) on a geodesic ball $B_R(o)$. If the constants $b,\, c,\, q$ and $r$ satisfy one of the following two conditions,
\begin{description}
\item[Case 1] $b\geq 0$, $c\geq0$ and $( q, r)\in U_1\cup U_2\cup U_3$, where
\begin{align}\begin{split}
U_1 =& \left\{( q, r): \quad \left(\frac{n+1}{n-1}-\frac{q}{p-1}\right)^2+\left(\frac{n+1}{n-1}-\frac{r}{p-1}\right)^2\leq \frac{4}{(n-1)^2}\right\},\\
U_2= &\left\{( q, r): \quad \frac{q}{p-1} \leq \frac{n+1}{n-1} , \quad  \frac{r}{p-1}<\frac{n+3}{n-1}\right\},\\
U_3= &\left\{( q, r): \quad\frac{q}{p-1}<\frac{n+3}{n-1},\quad \frac{r}{p-1} \leq \frac{n+1}{n-1}\right\};				
\end{split}		
\end{align}
\item[Case 2] $b\leq 0$, $c\leq 0$ and $( q, r)\in V_1\cup V_2\cup V_3$, where
\begin{align}\label{equa:1.9}
\begin{split}
V_1 = &\left\{( q, r):  \quad \left(\frac{n+1}{n-1}-\frac{q}{p-1}\right)^2+\left(\frac{n+1}{n-1}-\frac{r}{p-1}\right)^2\leq \frac{4}{(n-1)^2}\right\},\\
V_2=&\left\{( q, r): \quad\frac{q}{p-1} \geq \frac{n+1}{n-1},\quad  \frac{r}{p-1}>1\right\},\\
V_3=&\left\{( q, r):  \quad\frac{q}{p-1}>1 ,\quad  \frac{r}{p-1} \geq \frac{n+1}{n-1} \right\},			
\end{split}		
\end{align}
\end{description}
then there holds true
\begin{align*}
\sup_{B_{\frac{R}{2}}(o)}\frac{|\nabla v|}{v}\leq C(n,p,q,r) \frac{1+\sqrt{\kappa}R}{R} ,
\end{align*}
where the constant $C=C(n,p,q,r)$ depends only on $n,\,p,\,q$ and $r$ and may differ in different $U_i$ and $V_i$.
\end{cor}

\begin{rem}
To our best knowledge,  \thmref{t1} is the first Cheng-Yau type gradient estimate for  equation (\ref{equ0}) on Riemannian manifolds with Ricci curvature bounded from below. \thmref{t1} can cover some previous results .
\begin{enumerate}
\item When $b=c=0$, there are no restrictions on $q, r$ and  \thmref{t1} can recover Wang-Zhang's results(\cite{MR2880214}).

\item When $b=0$, the range of $r$ for the gradient estimate of $\Delta_pv+cv^r=0$ is $r<\frac{n+3}{n-1}(p-1)$ in the case $c>0$; or $r>p-1$ in the case $c<0$, which recover the main results of \cite{he2023gradient}.
	
\item When $c=0$, the region for the gradient estimate of $\Delta_pv+b|\nabla v|^q=0$, if $b>0$ and $r<\frac{n+3}{n-1}(p-1)$; or, $b<0$ and $r>p-1$, which recover a special case of results due to He-Hu-Wang (\cite{he2023nashmoser}).
\end{enumerate}
\end{rem}

An important application of \thmref{t1} is the following Harnack inequality.
\begin{thm}\label{t3}
Let $(M,g)$ be a complete non-compact Riemannian manifold with $\ric_g\geq-(n-1)\kappa g$, where $\kappa$ is a non-negative constant, and $p>1$. Suppose the constants $b,\, c, \, q$, and $r$ satisfying  the conditions in \thmref{t1}. If $v\in C^1(B_R(o))$ is a positive solution of equation \eqref{equ0} defined in a geodesic ball $B_R(o)\subset M$, then for any $x, y\in B_{R/2}(o)$, there holds
	$$
	v(x)/v(y)\leq  e^{C(n,p,q, r)(1+\sqrt{\kappa}R)}.
	$$
	If $v\in C^1(M)$ is a global positive solution of equation \eqref{equ0}, then for any $x,\, y\in M$, there holds
	$$
	v(x)/v(y)\leq  e^{C(n,p,q,r)\sqrt{\kappa}d(x,y)},
	$$
where $d(x,y)$ is the geodesic distance between $x$ and $y$.
\end{thm}

\begin{thm}\label{t4}
Let $(M, g)$ ($\dim(M)\geq 2$) be a non-compact complete manifold with with $\mathrm{Ric}_g \geq 0$ and $\Omega\subset M$ be
a domain.

(1). If $1<p<n$ and
$$p - 1 < r < \min\left\{\frac{n(p-1)}{n-p},\,\, \frac{n + 3}{n-1}(p-1)\right\},$$
then, for any positive solution $v$ to $\Delta_pv + v^r =0$ in $\Omega$ and any $o\in\Omega$ such that $d(o, \partial\Omega) > 2R$, there holds
$$\sup_{B_R(o)}v \leq C(n,p,r)R^{-\frac{p}{r-p+1}}.$$

(2). If $1<p<n$, $(q,r)\in U_1\cup U_2\cup U_3$ and
$$p - 1 < r < \frac{n(p-1)}{n-p},$$
then for any positive solution $v$ to the equation $\Delta_pv + |\nabla u|^q + v^r =0$ in $\Omega$ and any $o\in\Omega$ such that
$d(o, \partial\Omega)>2R$, there holds
$$\sup_{B_R(o)}v \leq C(n,p,q, r)R^{-\frac{p}{r-p+1}}.$$
\end{thm}

\begin{cor}\label{cc5}
Let $(M, g)$ be a non-compact complete manifold with with $\mathrm{Ric}_g \geq 0$ and $\Omega\subset M$ be a bounded domain.

(1). Assume the same condition is given as (1) of the above \thmref{t4}. If $v$ is a positive solution of $\Delta_pv + v^r =0$ on $\Omega$, then we have
$$v(x) \leq C(n,p,r)d(x, \partial\Omega)^{-\frac{p}{r-p+1}}.$$
If $v$ is a positive solution of $\Delta_pv + v^r =0$ in a punctured domain $\Omega\setminus\{o\}$, then for any $x$ near $o$, we have
$$v(x) \leq C(n,p,r)d(x, o)^{-\frac{p}{r-p+1}}.$$

(2). Assume the same condition is given as in (2) of the above \thmref{t4}. If $v$ is a positive solution of $\Delta_pv + |\nabla u|^q + v^r =0$ on $\Omega$, then we have
$$v(x) \leq C(n, p, q, r)d(x, \partial\Omega)^{-\frac{p}{r-p+1}}.$$
If $v$ is a positive solution of $\Delta_pv + |\nabla u|^q + v^r =0$ in a punctured domain $\Omega\setminus\{o\}$, then for any $x$ near $o$, we have
$$v(x) \leq C(n,p,q,r)d(x, o)^{-\frac{p}{r-p+1}}.$$
\end{cor}

\begin{rem}
The above \thmref{t4} and \corref{cc5} extend partially the corresponding results due to Dancer \cite{Dancer} and Pol\'{a}\v{c}ik, Quittner and Souplet \cite{MR2350853} for the case Euclidean space to the case $(M, g)$ is a complete non-compact Riemannian manifold with $\mathrm{Ric}\geq 0$. For more details we refer to Section 4 of the present paper.
\end{rem}

\subsection{Applications to singularity analysis on $\mathbb R^n$}\

The completeness of the Riemannian manifold and the lower boundedness of Ricci curvature are necessary for the application of Saloff-Coste's Sobolev inequalities (\cite{saloff1992uniformly}, see \lemref{salof}).  On $\bR^n$, the lower boundedness of Ricci curvature and Sobolev inequality always hold, so our theorem still holds. Hence, \thmref{t1} is of useful implications for the asymptotic behavior of solutions near
isolated singularities, see the corollaries below.
\begin{cor}\label{cor14}
Let $\Omega\subset \mathbb R^n$ be a domain. Assume $b,\, c,\, p,\, q$ and $r$ satisfy the conditions in \thmref{t1}. If $v\in C^2(\Omega)$ is a positive solution of
	\begin{align}\label{eq3.56}
		\begin{cases}
			\Delta_pv+b|\nabla v|^q+cv^r=0, \quad\text{in} \quad\Omega;\\
			p>1,\, b,\, c, \,q,\, r \in \bR,
		\end{cases}
	\end{align}
then for any $x\in\Omega$, we have
	\begin{align}\label{eq3.57}
		|\nabla\ln v(x)|\leq C(n,p,q,r)d(x, \partial\Omega)^{-1}.
	\end{align}
\end{cor}
	
\begin{rem}
Equation \eqref{sun1} is a special case of \eqref{eq3.56}, so \eqref{eq3.57} also holds for positive solutions of \eqref{sun1}.
Although \corref{cor14} and  \cite[Theorem 1.9]{filippucci2022priori} are derived by different methods,  they have some formal relations. If we rewrite the estimate \eqref{sun2} by (the constant a is absorbed by $c(n,p,q,r)$)
\begin{align*}
|u^{a-1}\nabla v (x)|\leq C(n,p, q,r)\left(d(x,\partial\Omega) \right)^{-\frac{pa}{r-p+1}-1}, \quad\forall x\in\Omega,
\end{align*}
then as $a\to 0^{+}$, the above estimate tends to the estimate in \corref{cor14}. However, our results can be established in a more wider range of $(q,r)$ than that in \cite[Theorem 1.9]{filippucci2022priori} and  \cite[Theorem D]{MR4150912}. The following figure shows the range of $(q,\, r)$ of Case 1 in \thmref{t1} and \cite[Theorem 1.9]{filippucci2022priori} with $n=5$ and $p=3$.
\end{rem}

	\begin{figure}[h]
	\begin{tikzpicture}
		\draw[help lines, color=red!5, dashed] (-0.5,-0.5) grid (4,4);
		\path[fill=green!30](-0.5, 3)--(4,3)--(4,-0.5)--(-0.5, -0.5);
		\path[fill=green!30](-0.5, 2)--(3,2)--(3,4)--(-0.5, 4);
		\fill[fill=green!30](3,3) circle (1);
		\path[pattern=north east lines] (2, 2)--(14/5, 2)--(14/5, 4)--(2 ,4);
		\draw[dotted] (4 ,-0.5) --(4, 2)  node[anchor=north]{$\frac{q}{p-1}= \frac{n+3}{n-1}$} -- (4 ,4) -- (4 ,4.5);
		\draw[dotted] (14/5,-0.5) --(14/5, 1)  node[anchor=north]{$\frac{q}{p-1}= \frac{n+2}{n}$} -- (14/5,4) -- (14/5,4.5);
		\draw[dotted] (-0.5, 4 )--(2, 4) node[anchor=east]{$\frac{r}{p-1}= \frac{n+3}{n-1}$}-- (4, 4 )  -- (6, 4);
		\filldraw[black] (3,3) circle (1pt) node[anchor=west]{$(\frac{n+1}{n-1}, \frac{n+1}{n-1})$};
		\filldraw[black] (2, 2) circle (1pt) node[anchor=east]{$(1,1)$};
		\draw[->,ultra thick] (-0.5,0)--(6,0) node[right]{$\frac{q}{p-1}$};
		\draw[->,ultra thick] (0,-0.5)--(0,4.5) node[above]{$\frac{r}{p-1}$};
	\end{tikzpicture}			
    \caption{The shadow region with north east lines represents the $(\frac{q}{p-1}, \frac{r}{p-1})$ in \cite[Theorem 1.9]{filippucci2022priori} and  \cite[Theorem D]{MR4150912}; the region filled with green color is our region of $(\frac{q}{p-1}, \frac{r}{p-1})$ for gradient estimate.}
    \label{fig1}
	\end{figure}
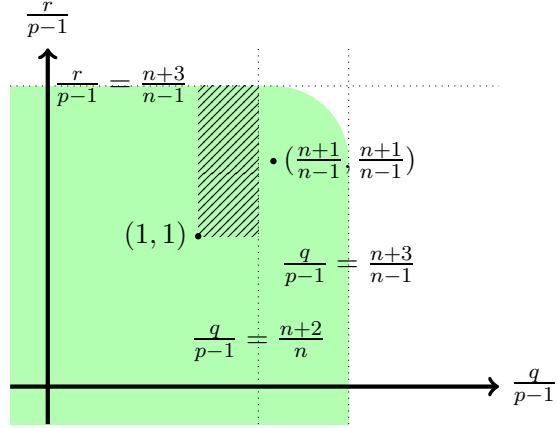

We can also employ the above \corref{cor14} to study the singularity of the positive solutions of \eqref{equ0} in domains of $\bR^n$.
\begin{cor}[Isolated singularity I]\label{cor15}
Assume that $b,\, c,\, p, \, q$ and $r$ satisfy the conditions in \thmref{t1}. Let $\Omega$ be a domain containing $0$ and $R>0$ be such that $d(0,\partial\Omega) \geq2R$. Then, for any solution $v\in C^2(\Omega\setminus\{0\})$ of \eqref{eq3.56} in $\Omega\setminus\{0\}$ there holds
\begin{align}\label{eq3.55}
v(x)\leq \sup_{|y|=R} v(y)\cdot \left(\frac{R}{|x|}\right)^{C(n,p,q,r)}, \quad\forall x\in B_{R}\setminus \{0\},
\end{align}
where $C(n,p,q,r)$ is a constant in \eqref{gradient} independent of $v$.
\end{cor}

\begin{cor}[Isolated singularity II*]\label{cor15*}
(i). Assume that $b\geq0$, $c\geq0$, $p=2$, $q$ and $r$ satisfy the conditions in \thmref{t1}. Let $\Omega\subset\mathbb{R}^n$ be a domain containing $0$ and $R>0$ be such that $d(0,\partial\Omega) \geq 2R$. Then, for any positive solution $v\in C^2(\Omega\setminus\{0\})$ of \eqref{eq3.56} in $\Omega\setminus\{0\}$ there hold true
$$
v(x)\leq C|x|^{2-n}\,\, \mbox{for}\,\, n\geq 3\quad\mbox{and}\quad v(x)\leq C\log\frac{1}{|x|}\,\,\mbox{for}\,\, n=2
$$
in a neighborhood of $0$, where $C(n,p,q,r,v)$ is a positive constant depending on $v$.

(ii). Assume that $b\leq0$, $c\leq0$, $p=2$, $q$ and $r$ satisfy the conditions in \thmref{t1}. Let $\Omega\subset\mathbb{R}^n$ be a domain containing $0$ and $R>0$ be such that $d(0,\partial\Omega) \geq 2R$. Then, for any solution $v\in C^2(\Omega\setminus\{0\})$ of \eqref{eq3.56} in $\Omega\setminus\{0\}$ there hold true
$$
v(x)\geq C|x|^{2-n}\,\, \mbox{for}\,\, n\geq 3\quad\mbox{and}\quad v(x)\geq C\log\frac{1}{|x|}\,\,\mbox{for}\,\, n=2
$$
in a neighborhood of $0$, where $C(n,p,q,r,v)$ is a positive constant depending on $v$.
\end{cor}

\begin{cor}[Boundary singularity]\label{cor16}
Assume that $b$, $c$, $p$, $q$ and $r$ satisfy the conditions in \thmref{t1}. Let $\Omega\subset\mathbb{R}^n$ be a domain with a $C^2$-smooth boundary $\partial\Omega$ and $v\in C^2(\Omega)$ be a positive solution of \eqref{eq3.56} in $\Omega$. Then there exists some $\delta>0$ such that for any $x\in\Omega_\delta$, where
	$\Omega_\delta = \{z\in\Omega:d(z,\partial\Omega)\leq \delta\}$,
we have
\begin{align*}
v(x)\leq \sup_{d(z^*,\partial\Omega)=\delta^*}  v(z^*)\left(\frac{\delta^*}{d(x,\partial\Omega)} \right)^{C(n,p,q,r)}.
\end{align*}
\end{cor}

\subsection{Liouville Theorem}\

As a consequence of \thmref{t1}, we can obtain the following Liouville theorem.
\begin{cor}\label{c14}
Let $(M,g)$ be a complete non-compact Riemannian manifold satisfying $\mathrm{Ric}_g\geq0$. Suppose the constants $b,\, c,\, q$ and $r$ satisfy the conditions given in \thmref{t1}. Then any positive entire solution of equation \eqref{equ0} on $M$ must be a constant.
\end{cor}
	
By taking a similar argument with Mitidieri and Pohozaev in \cite[Theorem 15.2]{MR1879326}, we can extend \cite[Theorem 15.2]{MR1879326} to the case the underlying domain is a non-compact complete Riemannian manifold with non-negative Ricci curvature.
\begin{prop}\label{p5}
Let $(M,g)$ be a complete non-compact Riemannian manifold with $\mathrm{Ric}\geq 0$. If $1<p<n$, then the problem
$$
\begin{cases}
\Delta_pv+|\nabla v|^q + v^r\leq 0, &x\in M;\\
v> 0, &x\in M,
\end{cases}
$$
with
$$
1<\frac{r}{p-1}<\frac{n}{n-p}\quad\text{or}\quad 1<\frac{q}{p-1}<\frac{n}{n-1}
$$
does not admit any positive solution.
\end{prop}

The combination of \corref{c14} and \propref{p5} yields the following theorem.
\begin{thm}\label{t6}
Let $(M,g)$ be a complete non-compact Riemannian manifold with $\mathrm{Ric}\geq 0$. If $1<p<n$, and one of the following conditions is satisfied,
\begin{itemize}
			
\item $\frac{q}{p-1}\leq 1 \quad \mbox{and} \quad \frac{r}{p-1}<\max\left\{\frac{n+3}{n-1}, \frac{n}{n-p}\right\};$
			
\item $1<\frac{q}{p-1}<\frac{n}{n-1};$
			
\item $\frac{n+1}{n-1}\geq \frac{q}{p-1}\geq\frac{n}{n-1}\quad \mbox{and} \quad\frac{r}{p-1}<\max\left\{\frac{n+3}{n-1}, \frac{n}{n-p}\right\}$;
			
\item $\frac{n+1}{n-1}< \frac{q}{p-1}<\frac{n+3}{n-1}\quad \mbox{and} \quad \frac{r}{p-1}<\max\left\{\frac{n+1}{n-1}+\sqrt{\frac{4}{(n-1)^2}-\left(\frac{n+1}{n-1}-\frac{r}{p-1}\right)^2}, \frac{n}{n-p}\right\}$;
			
\item $\frac{q}{p-1}\geq \frac{n+3}{n-1}\quad \mbox{and} \quad 1<\frac{r}{p-1}<\frac{n}{n-p}$,
\end{itemize}
then the equation
	$$
	\Delta_pv + |\nabla v|^q + v^r=0
	$$
does not admit any global positive solution on $M$.
\end{thm}

Figure \ref{fig2} shows the region of $(\frac{q}{p-1}, \frac{r}{p-1})$ with $n=5, p=3$ in \thmref{t6}.
\begin{figure}[h]
\begin{tikzpicture}
			\draw[help lines, color=red!5, dashed] (-0.5,-0.5) grid (4,4);
			\path[fill=green!30](-0.5, 5)--(2,5)--(2,-0.5)--(-0.5, -0.5);
			\path[fill=green!30](2, -0.5)--(2.5, -0.5)--(2.5, 6)--(2, 6);
			\path[fill=green!30](2.5, -0.5)--(4, -0.5)--(4, 5)--(2.5 ,5);
			\path[fill=green!30](4, 2)--(4, 5)--(7, 5)--(7,2);
			\draw  (2.5 ,6) --(2.5, 5) -- (7,5);
			\draw  (4, -0.5)--(4, 2)-- (7,2);
			\draw (-0.5, 5 )--(2, 5) -- (2, 6);
			\filldraw[black] (2, 5) circle (0.5pt) node[anchor=north]{$(1,\frac{n}{n-p})$};
			\filldraw[black] (4, 2) circle (0.5pt) node[anchor=north]{$(\frac{n+3}{n-1}, 1)$};
			\filldraw[black] (2, 4) circle (0.5pt) node[anchor=north]{$(1, \frac{n+3}{n-1} )$};
			\draw[->,ultra thick] (-0.5,0)--(7,0) node[right]{$\frac{q}{p-1}$};
			\draw[->,ultra thick] (0,-0.5)--(0,6) node[above]{$\frac{r}{p-1}$};
		\end{tikzpicture}	
		\caption{The region of $(\frac{q}{p-1}, \frac{r}{p-1})$ in \thmref{t6}.}
		\label{fig2}
\end{figure}
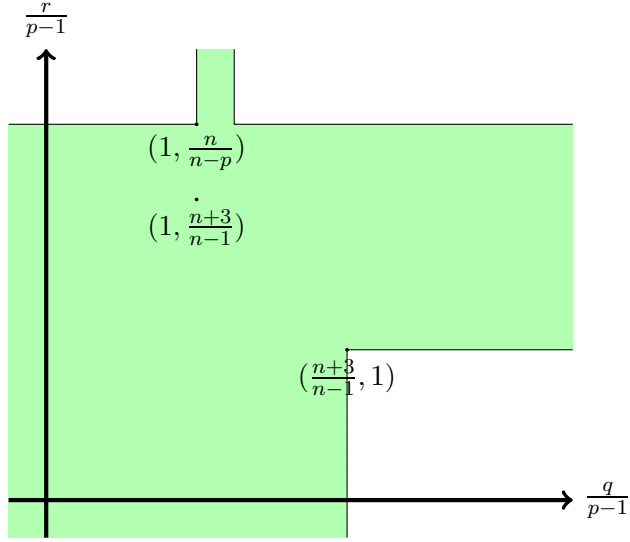
	
\subsection{Global gradient estimate}\

If $v$ is an entire solution of \eqref{equ0} on a complete non-compact Riemannian manifold with $\mathrm{Ric}_g\geq -(n-1)\kappa$, then, by \thmref{t1} we have
$$\frac{|\nabla v|}{v}\leq C(n,p,q,r)\sqrt{\kappa}.$$
In this paper, motivated by \cite{MR3275651} we will also study the bound of constant $c(n,p,q,r)$ in different regions.
\begin{thm}\label{thm1.10}
Let $(M,g)$ be an $n$-dim($n\geq 2$) complete Riemannian manifold satisfying $\mathrm{Ric}_g\geq-(n-1)\kappa g$ for some constant $\kappa\geq0$. Let $p>1$ and $v\in C^1(M)$ be a global positive solution of \eqref{equ0}. Then we have
\begin{enumerate}
\item if $(b, c, q, r)\in W_1$, there holds true
			$$
		\frac{|\nabla v|}{v}\leq \frac{(n-1)\sqrt{\kappa}}{p-1}.
		$$
		\item if  $(b, c, q, r)\in W_2$, there holds true
		$$
		\frac{|\nabla v|}{v}\leq \frac{2\sqrt{\kappa}}{(p-1)\sqrt{\left(\frac{n+3}{n-1}-\frac{r}{p-1}\right)\left(\frac{r}{p-1}-1\right)}}.
		$$
			
\item if $(b, c, q, r)\in W_3$, there holds true
		$$
		\frac{|\nabla v|}{v}\leq \frac{2\sqrt{\kappa}}{(p-1)\sqrt{\left(\frac{n+3}{n-1}-\frac{q}{p-1}\right)\left(\frac{q}{p-1}-1\right)}}.
		$$
\item if $(b, c, q, r)\in W_4$, there holds true
		$$
		\frac{|\nabla v|}{v}\leq \frac{2\sqrt{\kappa}}{(p-1)\sqrt{\frac{4}{(n-1)^2}-\left( \frac{n+1}{n-1}-\frac{q}{p-1}\right)^2
				-\left(\frac{n+1}{n-1}-\frac{r}{p-1}\right)^2}}.
		$$
\end{enumerate}
\end{thm}

\thmref{thm1.10} implies some special results in the case of $b=0$ or $c=0$. Here, we only consider the case $b=0$.
\begin{cor}\label{cor1.11}
Let $(M,g)$ be an $n$-dim($n\geq 2$) complete Riemannian manifold satisfying $\mathrm{Ric}_g\geq-(n-1)\kappa g$ for some constant $\kappa\geq0$. Let $p>1$ and $v\in C^1(M)$ be a global positive solution of generalized Lane-Emden equation
		$$
		\Delta_pv+av^q =0.
		$$
\begin{enumerate}
\item If $a$, $p$ and $q$ satisfy $a\left(\frac{n+1}{n-1} -\frac{q}{p-1}\right)\geq0$, then we have
			$$
			\frac{|\nabla v|}{v}\leq \frac{(n-1)\sqrt{\kappa}}{p-1}.
			$$
\item If $p$ and $q$ satisfy $\left|\frac{n+1}{n-1}-\frac{q}{p-1}\right|<\frac{2}{n-1}$, then we have
$$
\frac{|\nabla v|}{v}\leq \frac{2\sqrt{\kappa}}{(p-1)\sqrt{\left(\frac{n+3}{n-1}-\frac{q}{p-1}\right)\left(\frac{q}{p-1}-1\right)}}.
$$
\end{enumerate}
\end{cor}

\begin{rem}
If we choose $a=0$ in \corref{cor1.11}, we recover the sharp gradient estimate for p-Laplace equation due to Sung and Wang \cite[Theorem 2.2]{MR3275651}.
\end{rem}
	
\begin{rem}
The first item in \corref{cor1.11} also improved  \cite[Theorem 1.1]{MR4594369} in the following aspects.
\begin{itemize}
\item The condition sectional curvature $\mathrm{Sec}_g\geq -\kappa$ in \cite[Theorem 1.1]{MR4594369} is weakened to $\mathrm{Ric}_g\geq -(n-1)\kappa$;
		
\item Their range of gradient estimate is
		$$
		a\left(\frac{n+2}{n}(p-1)-q\right)\geq 0,
		$$
while our range is $a\left(\frac{n+1}{n-1}(p-1)-q\right)\geq 0$ which is wider since $\frac{n+1}{n-1}>\frac{n+2}{n}$.
		
\item Their bound for gradient
		$$
		\frac{|\nabla v|}{v}\leq \frac{\sqrt{n(n-1)\kappa}}{p-1}
		$$
is larger than ours since $n-1<n$.
\end{itemize}
\end{rem}

\textbf{Main ideas:} It is well-known that the Nash-Moser iteration is a powerful and fundament tool in the theory of partial differential equations and geometric analysis (see \cite{he2023gradient, wang2021gradient1, MR2880214, Wang}). As we did in \cite{he2023gradient}, we first employ a logarithm transformation $u=-(p-1)\ln v$ where $v$ is the positive solution to equation \eqref{equ0}, and then consider $\mL(f^\alpha)$ where $f^\alpha=|\nabla u|^{2\alpha}$ and $\mL$ is the linearized operator $p$-Laplace operator at $u$. We need to look carefully for some regions of $(\frac{q}{p-1}, \frac{r}{p-1})$ such that some required differential inequalities on $f^\alpha$ can be established for these regions, furthermore we can use Saloff-Coste's Sobolev inequality to obtain a key integral estimate on $f$ which is universal. By the integral estimate we take a standard Nash-Moser iteration to get $L^\infty$-estimate of $f$.

The paper is organized as follows. In Section 2 we give some fundamental definitions and some notations which will be used later, recall the Saloff-Coste's Sobolev inequality on a Riemannian manifold, and establish two important lemmas which give the pointwise estimates on $|\nabla u|^2$. In Section 3 we give the proofs of \thmref{t1} and \thmref{t3}. We provide the proof of \thmref{t6} in Section 4. The proof of \thmref{thm1.10} and its corollary will be shown in Section 5.
	
\section{Preliminaries}
Throughout this paper, if not specified, we always assume that the $n$-dim Riemannian manifold $(M,g)$ is non-compact complete with $\mathrm{Ric}_g\geq-(n-1)\kappa g$ and $\dim(M)\geq 2$. Let $\nabla$ be the corresponding Levi-Civita connection. For any function $\varphi\in C^1(M)$, we denote $\nabla \varphi\in \Gamma(T^*M)$ by $\nabla \varphi(X)=\nabla_X\varphi$. We denote the volume form $$d\vol=\sqrt{\det(g_{ij})}d x_1\wedge\ldots\wedge dx_n,$$
where $(x_1,\ldots, x_n)$ is a local coordinate chart, and for simplicity, we may omit the volume form in integrals over $M$.
	
For any $p>1$, the $p$-Laplace operator is defined by
	$$
	\Delta_pu=\di(|\nabla u|^{p-2}\nabla u).
	$$
The solution of $p$-Laplace equation $\Delta_pu=0$ is the critical point of the energy functional
	$$
	E(u)=\int_M|\nabla u|^p.
	$$
	
\begin{defn}\label{def1}
A function $v$ is said to be a (weak) positive solution of equation \eqref{equ0} on a region $\Omega$ if $v\in C^1(\Omega), v>0 $ and for all $\psi\in C_0^\infty(\Omega)$, we have
\begin{align*}
-\int_M|\nabla v|^{p-2}\la\nabla v,\nabla\psi\ra+b\int_M|\nabla v|^q\psi+c\int_Mv^r\psi=0.
\end{align*}
\end{defn}
It is worth mentioning that any solution $v$ of equation (\ref{equ0}) satisfies $v\in W^{2,2}_{loc}(\Omega)$ and $v\in C^{1,\alpha}(\Omega)$ for some $\alpha\in(0,1)$ (for example, see \cite{MR0709038, MR0727034,MR0474389}). Moreover, away from $\{\nabla v=0\}$, $v$ is in fact smooth.
	
In our proof of \thmref{t1}, the following Sobolev inequality due to Saloff-Coste \cite{saloff1992uniformly} plays an important role:

\begin{lem}[Saloff-Coste]\label{salof}
Let $(M,g)$ be a complete manifold with $\mathrm{Ric}\geq-(n-1)\kappa$. For $n>2$, there exists $c_0$ depending only on $n$, such that for all $B\subset M$ of radius R and  volume $V$ we have for $f\in C^{\infty}_0(B)$
		$$
		\|f\|_{L^{\frac{2n}{n-2}}(B)}^2\leq e^{c_0(1+\sqrt{\kappa}R)}V^{-\frac{2}{n}}R^2\left(\int_B|\nabla f|^2+R^{-2}f^2\right).
		$$
For $n=2$, the above inequality holds with $n$ replaced by any fixed $n'>2$.
\end{lem}
	
Let $v$ be a positive solution of equation \eqref{equ0}, i.e.
	$$
	\Delta_pv+b|\nabla v|^p+cv^r=0.
	$$
Set $$u=-(p-1)\ln v.$$
Then $u$ satisfies
	\begin{align}\label{equ1}
		\Delta_pu-|\nabla u|^p-b_1e^{b_2u}|\nabla v|^q-c_1e^{c_2u}=0,
	\end{align}
where
	\begin{align}\label{const}
		b_1 =b(p-1)^{p-1-q},\quad b_2=\frac{p-1-q}{p-1}, \quad c_1 = c(p-1)^{p-1},\quad c_2=\frac{p-1-r}{p-1}.
	\end{align}
Let $f=|\nabla u|^2$. Then the above equation becomes
\begin{align}\label{equf}
\Delta_pu-f^{\frac{p}{2}}-b_1e^{b_2u}f^{\frac{q}{2}}-c_1e^{c_2u}=0.
\end{align}
	
Now we consider the linearized operator $\mL$ of $p$-Laplace operator at $u$. Direct computation shows
\begin{align*}
\mL(\psi)=\di\left(f^{p/2-1}\nabla \psi+(p-2)f^{p/2-2}\la\nabla \psi,\nabla u\ra\nabla u\right).
\end{align*}
We need  the following identity for the operator $\mathcal L(f^\alpha)$.
	
\begin{lem}[\cite{he2023gradient}]
For any $\alpha \geq 1$,
		\begin{align}
			\label{bochner1}
			\begin{split}
				\mathcal{L} (f^{\alpha}) =
				&\
				\alpha(\alpha+\frac{p}{2}-2)f^{\alpha+\frac{p}{2}-3}|\nabla f|^2
				+
				2\alpha f^{\alpha+\frac{p}{2}-2} (|\nabla\nabla u|^2 + \ric(\nabla u,\nabla u) )
				\\
				&
				+
				\alpha(p-2)(\alpha-1)f^{\alpha+\frac{p}{2}-4}\langle\nabla f,\nabla u\rangle^2
				+2\alpha f^{\alpha-1}\langle\nabla\Delta_p u,\nabla u\rangle
				.
			\end{split}
		\end{align}
holds point-wisely in $\{x:f(x)>0\}$.
\end{lem}
	
Using the above lemma, we give a basic estimate for $\mL(f^\alpha)$.
	
\begin{lem}\label{lem2.4}
Let $u\in C^1(\Omega)$ be a solution of \eqref{equ1}. If $bc \geq 0$, then there holds point-wisely on $\{x:f(x)\neq 0\}$,
		\begin{align}
			\label{equa2.5}
			\begin{split}
				\frac{f^{2-\alpha-\frac{p}{2}}}{2\alpha}\mathcal{L} (f^{\alpha})
				\geq &\
				\frac{f^2}{n-1}-(n-1)\kappa f
				-\frac{a_1}{2}f^{\frac{1}{2}}|\nabla f|
				+  \left( \frac{n+1}{n-1}-\frac{q}{p-1}\right)e^{b_2u}f^{\frac{q-p}{2}+2}
				\\
				&
				+c_1\left( \frac{n+1}{n-1}-\frac{r}{p-1}\right)e^{c_2u}f^{2-\frac{p}{2}}
				+\Big(\frac{1}{n-1} -\frac{d_1}{2\alpha-1}  \Big)\left(b_1e^{b_2u}f^{\frac{q-p}{2}+1}\right)^2
				\\
				& +\Big(\frac{1}{n-1}
				-\frac{d_2}{2\alpha-1}  \Big)\left(c_1e^{c_2u}f^{1-\frac{p}{2}}\right)^2,
				\end{split}
		\end{align}
where $d_1$ and $d_2$ are positive constants depending only on $n$ and $p$.
\end{lem}
	
\begin{proof}
Let $\{e_1,e_2,\ldots, e_n\}$ be an orthonormal frame of $TM$ on a domain with $f\neq 0$ such that $e_1=\frac{\nabla u}{|\nabla u|}$. We have $u_1 = f^{\frac{1}{2}}$ and
\begin{align}\label{add}
u_{11} = \frac{1}{2}f^{-\frac{1}{2}}f_1 = \frac{1}{2}f^{-1}\la\nabla u,\nabla f\ra.
\end{align}
It follows from (\ref{equf}) that
\begin{align}\label{equa:2.5}
\begin{split}
\la\nabla\Delta_p u,\nabla u\ra =&\ \la\nabla( f^{\frac{p}{2}}+b_1e^{b_2u}f^{\frac{q}{2}}+c_1e^{c_2u}),\nabla u\ra\\
=&\ pf^{\frac{p}{2} }u_{11}+b_1b_2e^{b_2u}f^{\frac{q}{2}+1} + qb_1 e^{b_2u}f^{\frac{q}{2} }u_{11}+c_1b_2e^{c_2u}f.
\end{split}
\end{align}
Using the fact $u_1 = f^{\frac{1}{2}}$ again, we have the identity
		\begin{align}
			\label{equ6}
			\sum_{i=1}^n |u_{1i}|^2 = \frac{|\nabla f|^2}{4f}.
		\end{align}
By (\ref{equ6}) and Cauchy inequality, we estimate the Hessian of $u$ by
		\begin{align}\label{equa:2.7}
			|\nabla\nabla u|^2
			\geq& \sum_{i=1}^n u_{1i}^2+\sum_{i=2}^n u_{ii}^2
			\geq u_{11}^2+\frac{1}{n-1}\Big(\sum_{i=2}^n u_{ii}\Big)^2.
		\end{align}
Substituting (\ref{add}), (\ref{equa:2.5}), (\ref{equ6}) and (\ref{equa:2.7}) into (\ref{bochner1}), we obtain
\begin{align}\label{equ211}
\begin{split}
\frac{f^{2-\alpha-\frac{p}{2}}}{2\alpha}\mathcal{L} (f^{\alpha})\geq &\
				(2\alpha-1)(p-1) u_{11}^2+ \ric(\nabla u,\nabla u)+\frac{1}{n-1}\Big(\sum_{i=2}^n u_{ii}\Big)^2\\
				&\
				+ f^{ 1-\frac{p}{2}}\left(pf^{\frac{p}{2} }u_{11}+b_1b_2e^{b_2u}f^{\frac{q}{2}+1}
				+ qb_1 e^{b_2u}f^{\frac{q}{2} }u_{11}
				+c_1b_2e^{c_2u}f\right).
\end{split}
\end{align}
		
On the region $\{x:f(x)\neq 0\}$, the $p$-Laplace operator can be expressed in terms of $f$ as
\begin{align}\label{equ:2.2}
\Delta_p u = &f^{\frac{p}{2}-1}\left((p-1)u_{11}+\sum_{i=2}^nu_{ii}\right).
\end{align}
Substituting (\ref{equ:2.2}) into equation (\ref{equf}) yields
\begin{align}\label{equ2}
\sum_{i=2}^nu_{ii}=f+ b_1e^{b_2u}f^{\frac{q-p}{2}+1}+c_1e^{c_2u}f^{1-\frac{p}{2}}-(p-1)u_{11}.
\end{align}
Expanding the square term directly, we obtain
\begin{align*}
			\Big(\sum_{i=2}^n u_{ii}\Big)^2
			=&\ f^2+ \left(b_1e^{b_2u}f^{\frac{q-p}{2}+1}\right)+\left(c_1e^{c_2u}f^{1-\frac{p}{2}}\right)^2+
			2b_1c_1e^{(b_2+c_2)u}f^{2-p+\frac{q}{2}}
			\\
			&
			+(p-1)^2u^2_{11}
			+2\left(b_1e^{b_2u}f^{\frac{q-p}{2}+2}+c_1e^{c_2u}f^{2-\frac{p}{2}}\right)
			-2(p-1)fu_{11}
			\\
			&-2(p-1)\left(b_1e^{b_2u}f^{\frac{q-p}{2}+1}+c_1e^{c_2u}f^{1-\frac{p}{2}}\right)u_{11}.
		\end{align*}
Noting $bc\geq 0$ implies $b_1c_1\geq 0$, so we can omit the two non-negative terms
$$2b_1c_1e^{(b_2+c_2)u}f^{2-p+\frac{q}{2}}\quad \text{and} \quad (p-1)^2u^2_{11}$$
to obtain
\begin{align*}
\Big(\sum_{i=2}^n u_{ii}\Big)^2\geq &\ f^2+ \left(b_1e^{b_2u}f^{\frac{q-p}{2}+1}\right)^2+\left(c_1e^{c_2u}f^{1-\frac{p}{2}}\right)^2-2(p-1)fu_{11}\\
			&+2\left(b_1e^{b_2u}f^{\frac{q-p}{2}+2}+c_1e^{c_2u}f^{2-\frac{p}{2}}\right)\\
			&-2(p-1)\left(b_1e^{b_2u}f^{\frac{q-p}{2}+1}+c_1e^{c_2u}f^{1-\frac{p}{2}}\right)u_{11}.
\end{align*}
Then it follows from the above inequality and \eqref{equ211} that
\begin{align}
\label{equ2.13}
\begin{split}
\frac{f^{2-\alpha-\frac{p}{2}}}{2\alpha}\mathcal{L} (f^{\alpha})\geq &\
				\left(2\alpha-1\right)(p-1) u_{11}^2 + \ric(\nabla u,\nabla u)
				+\left(p-\frac{2(p-1)}{n-1}\right)f u_{11}\\
				&+ b_1\left(b_2+\frac{2}{n-1}\right)e^{b_2u}f^{\frac{q-p}{2}+2}
				+\frac{1}{n-1}\Big(f^2+  \left(b_1e^{b_2u}f^{\frac{q-p}{2}+1}\right)^2\\
				&+\left(c_1e^{c_2u}f^{1-\frac{p}{2}}\right)^2
				-2(p-1)  c_1e^{c_2u}f^{1-\frac{p}{2}} u_{11}\Big)\\
				&+c_1\left(c_2+\frac{2}{n-1}\right)e^{c_2u}f^{2-\frac{p}{2}}
				+ b_1\left(q-\frac{2(p-1)}{n-1}\right) e^{b_2u}f^{\frac{q-p}{2}+1 }u_{11}.
\end{split}
\end{align}
By the basic inequality  $\mu^2+2\mu\nu\geq -\nu^2$, we obtain
\begin{align}\label{equ213}
\begin{split}
				&\frac{1}{2}\left(2\alpha-1\right)(p-1) u_{11}^2
				+ b_1\left(q-\frac{2(p-1)}{n-1}\right) e^{a_2u}f^{\frac{q-p}{2}+1 }u_{11}\\
				\geq & -\frac{d_1}{2\alpha-1}\left(b_1e^{a_2u}f^{\frac{q-p}{2}+1}\right)^2,\\
\end{split}
\end{align}	
and
\begin{align}\label{equ2.14}
			&\frac{1}{2}\left(2\alpha-1\right)(p-1) u_{11}^2
				-\frac{2(p-1)}{n-1}c_1 e^{b_2u}f^{1-\frac{p}{2} }u_{11}
				\geq -\frac{d_2}{2\alpha-1}\left(c_1e^{b_2u}f^{1-\frac{p}{2}}\right)^2,
\end{align}
where the constants $d_1$ and $d_2$ are given by
\begin{align*}
d_1=\frac{1}{2(p-1)}\left(q-\frac{2(p-1)}{n-1}\right)^2\quad\text{and}\quad d_2 = \frac{2(p-1)}{(n-1)^2}.
\end{align*}
Now, if we denote $a_1 = \left|p-\frac{2(p-1)}{n-1}\right|$, then by Ricci curvature bound and \eqref{add} we have
		$$
		\mathrm{Ric}(\nabla u, \nabla u)+\left(p-\frac{2(p-1)}{n-1}\right)f u_{11}\geq -2(n-1)\kappa f-\frac{a_1}{2}f^{\frac{1}{2}}|\nabla f|.
		$$
It follows from \eqref{equ2.13}, \eqref{equ213}, \eqref{equ2.14} and the above equality that
\begin{align*}
\frac{f^{2-\alpha-\frac{p}{2}}}{2\alpha}\mathcal{L} (f^{\alpha})\geq &\
			\frac{f^2}{n-1}-(n-1)\kappa f -\frac{a_1}{2}f^{\frac{1}{2}}|\nabla f|
			+ \left(b_2+\frac{2}{n-1}\right)e^{b_2u}f^{\frac{q-p}{2}+2}\\
			&
			+c_1\left(c_2+\frac{2}{n-1}\right)e^{c_2u}f^{2-\frac{p}{2}}
			+\Big(\frac{1}{n-1} -\frac{d_1}{2\alpha-1}  \Big)\left(b_1e^{b_2u}f^{\frac{q-p}{2}+1}\right)^2\\
			& +\Big(\frac{1}{n-1}-\frac{d_2}{2\alpha-1}  \Big)\left(c_1e^{c_2u}f^{1-\frac{p}{2}}\right)^2.
\end{align*}
Substituting $b_2=\frac{p-1-q}{p-1}$ and $c_2=\frac{r-1-q}{p-1}$ into the above estimate, we obtain the required estimate. Thus, we finish the proof.
\end{proof}
	
Since $d_1$ and $d_2$ are both constants depending on $n, \, p$ and $q$, there exists some constant $\alpha_1>0$ large enough such that, for any $\alpha\geq \alpha_1$, there hold
	$$
	\frac{1}{n-1} -\frac{d_1}{2\alpha-1}> 0\quad\text{and}\quad \frac{1}{n-1} -\frac{d_2}{2\alpha-1} > 0.
	$$
From now on, we always assume that $\alpha\geq\alpha_1$.
	
\begin{lem}\label{lem2.5}
Let $u\in C^1(\Omega)$ be a solution of \eqref{equ1}. If $bc\geq 0$ and $(b, c,  q, r)\in W_1\cup W_2\cup W_3\cup W_4$, where
		\begin{align*}
			W_1=&\left\{(b, c,  q, r): b\left(\frac{n+1}{n-1}-\frac{q}{p-1}\right)\geq 0 \quad \text{and}\quad
			c\left(\frac{n+1}{n-1}-\frac{r}{p-1}\right)\geq 0 \right\}	 ;
			\\
			W_2=&\left\{(b, c,  q, r): b\left(\frac{n+1}{n-1}-\frac{q}{p-1}\right)\geq 0 \quad \text{and}\quad
			\left|\frac{n+1}{n-1}-\frac{r}{p-1}\right|<\frac{2}{n-1} \right\};
			\\
			W_3=&\left\{(b, c, q, r):
			\left|\frac{n+1}{n-1}-\frac{q}{p-1}\right|<\frac{2}{n-1}  \quad \text{and}\quad
			c\left(\frac{n+1}{n-1}-\frac{r}{p-1}\right)\geq 0 \right\};
			\\
			W_4=&\left\{(b, c, q, r):
			\left( \frac{n+1}{n-1}-\frac{q}{p-1}\right)^2
			+\left(\frac{n+1}{n-1}-\frac{r}{p-1}\right)^2
			<\frac{4}{(n-1)^2}\right\},
\end{align*}
then there exists large enough $\alpha_0$ and some $\beta_0>0$ such that
\begin{align}\label{equa2.20}
\mathcal{L} (f^{\alpha_0})\geq &\ 2\alpha_0\beta_0f^{ \alpha_0+\frac{p}{2}}
			-2(n-1)\alpha_0f^{\alpha_0+\frac{p}{2}-1}
			-\alpha_0a_1f^{ \alpha_0+\frac{p-3}{2}}|\nabla f|,
\end{align}
where $\alpha_0$ and $\beta_0$ depend only on $n,\, p,\, q$ and $r$. The definitions of $\alpha_0$ and $\beta_0$ may be different from each other in $W_1$, $W_2$, $W_3$ and $W_4$.
\end{lem}
	
\begin{proof}
We use $R_i$ to denote the $i$-th term in the right hand side of the estimate in \eqref{equa2.5}.  Since $\alpha\geq\alpha_1$, it is easy to see that $R_6$ and $R_7$ are non-negative.
\begin{enumerate}
\item If $(b, c, q, r)\in W_1$, then $R_4$ and $R_5$ are non-negative. Omitting non-negative terms $R_4$, $R_5$, $R_6$ and $R_7$ yields
\begin{align}\label{w1}
\frac{f^{2-\alpha-\frac{p}{2}}}{2\alpha}\mathcal{L} (f^{\alpha})\geq\frac{f^2}{n-1}-(n-1)\kappa f-\frac{a_1}{2}f^{\frac{1}{2}}|\nabla f|.
\end{align}
\item If $(b, c,  q, r)\in W_2$, then $R_4$ is non-negative. By discarding $R_4$ and $R_6$, we obtain
\begin{align*}
\frac{f^{2-\alpha-\frac{p}{2}}}{2\alpha}\mathcal{L} (f^{\alpha})\geq &\ \frac{f^2}{n-1}-(n-1)\kappa f
+\left(p-\frac{2(p-1)}{n-1}\right)f u_{11}\\
&+c_1\left( \frac{n+1}{n-1}-\frac{r}{p-1}\right)e^{c_2u}f^{2-\frac{p}{2}}\\
&+\Big(\frac{1}{n-1}-\frac{d_2}{2\alpha-1}\Big)\left(c_1e^{c_2u}f^{1-\frac{p}{2}}\right)^2.
\end{align*}
By the fact $\mu^2+2\mu\nu\geq -\nu^2$, we have
\begin{align*}
&c_1\left(\frac{n+1}{n-1}-\frac{r}{p-1}\right)e^{c_2u}f^{2-\frac{p}{2}}+\Big(\frac{1}{n-1} -\frac{d_2}{2\alpha-1} \Big)\left(c_1e^{c_2u}f^{1-\frac{p}{2}}\right)^2\\
\geq &-f^2\frac{\left(\frac{n+1}{n-1}-\frac{r}{p-1}\right)^2}{\frac{4}{n-1} -\frac{4d_2}{2\alpha-1}}.
\end{align*}
It follows
			\begin{align*}
				\frac{f^{2-\alpha-\frac{p}{2}}}{2\alpha}\mathcal{L} (f^{\alpha})
				\geq &\
				\left(\frac{1}{n-1}-\frac{\left(\frac{n+1}{n-1}-\frac{r}{p-1}\right)^2}{\frac{4}{n-1} -\frac{4d_2}{2\alpha-1} }\right)f^2-(n-1)\kappa f
				-\frac{a_1}{2}f^{\frac{1}{2}}|\nabla f|.
			\end{align*}
Note that the condition $\left|\frac{n+1}{n-1}-\frac{r}{p-1}\right|<\frac{2}{n-1}$ implies
$$
\lim_{\alpha\to\infty}\frac{1}{n-1}-\frac{\left(\frac{n+1}{n-1}-\frac{r}{p-1}\right)^2}{\frac{4}{n-1} -\frac{4d_2}{2\alpha-1} }= \frac{1}{n-1}- \frac{n-1}{4}\left(c_2+\frac{2}{n-1}\right)^2 >0.
$$
So we can choose $\alpha_2$ large enough such that, for $\alpha\geq \alpha_2$, there holds true
$$
B^1_{n, p, q, r,\alpha}:=\frac{1}{n-1}-\frac{\left(\frac{n+1}{n-1}-\frac{r}{p-1}\right)^2}{\frac{4}{n-1} -\frac{4d_2}{2\alpha-1} }>0.
$$
It follows that
\begin{align}\label{w2}
\frac{f^{2-\alpha-\frac{p}{2}}}{2\alpha}\mathcal{L} (f^{\alpha})\geq &\ B^1_{n, p, q, r,\alpha}f^2-(n-1)\kappa f
-\frac{a_1}{2}f^{\frac{1}{2}}|\nabla f|.
\end{align}
			
\item If $(b, c, q, r)\in W_3$,  then $R_6$  is non-negative. Note that $q$ and $r$ are in relatively symmetric positions.
If we repeat the procedure in  the above item, we can infer that there exists $\alpha_3>0$ such that for $\alpha\geq\alpha_3$,
\begin{align}\label{w3}
\frac{f^{2-\alpha-\frac{p}{2}}}{2\alpha}\mathcal{L} (f^{\alpha})
\geq &\ B^2_{n,p,q,r,\alpha}f^2-(n-1)\kappa f -\frac{a_1}{2}f^{\frac{1}{2}}|\nabla f|
\end{align}
where
$$
B^2_{n, p, q, r,\alpha}:=\frac{1}{n-1}-\frac{\left(\frac{n+1}{n-1}-\frac{q}{p-1}\right)^2}{\frac{4}{n-1} -\frac{4d_1}{2\alpha-1} }>0
$$
is a positive constant and depends only on $n,\,p,\,q,\,r$ and $\alpha$.

\item In the case $(b, c, q, r)\in W_4$, by using the inequality $x^2+2xy\geq -y^2$ twice, we have
\begin{align*}
&b_1\left(\frac{n+1}{n-1}-\frac{q}{p-1}\right)e^{b_2u}f^{\frac{q-p}{2}+2}+\Big(\frac{1}{n-1} -\frac{d_1}{2\alpha-1}\Big)\left(b_1e^{b_2u}f^{\frac{q-p}{2}+1}\right)^2\\
\geq& - f^2\frac{\left(\frac{n+1}{n-1}-\frac{q}{p-1}\right)^2}{\frac{4}{n-1} -\frac{d_1}{2\alpha-1}},\\
\mbox{and}\\
&c_1\left(\frac{n+1}{n-1}-\frac{r}{p-1}\right)e^{c_2u}f^{2-\frac{p}{2}}+\Big(\frac{1}{n-1} -\frac{d_2}{2\alpha-1}  \Big)\left(c_1e^{c_2u}f^{1-\frac{p}{2}}\right)^2\\
\geq & -f^2\frac{\left(\frac{n+1}{n-1}-\frac{r}{p-1}\right)^2}{\frac{4}{n-1} -\frac{d_2}{2\alpha-1}}.
\end{align*}
It follows from the above two inequalities that
			\begin{align}\label{w4}
				\frac{f^{2-\alpha-\frac{p}{2}}}{2\alpha}\mathcal{L} (f^{\alpha})
				\geq &\
				B^3_{n, p, q, r,\alpha}f^2-(n-1)\kappa f
				-\frac{a_1}{2}f^{\frac{1}{2}}|\nabla f|,
			\end{align}
where
$$
B^3_{n, p, q, r,\alpha}:=\frac{1}{n-1}-\frac{\left(\frac{n+1}{n-1}-\frac{q}{p-1}\right)^2}{\frac{4}{n-1} -\frac{4d_1}{2\alpha-1}}
-\frac{\left(\frac{n+1}{n-1}-\frac{r}{p-1}\right)^2}{\frac{4}{n-1} -\frac{4d_2}{2\alpha-1}}.
$$
We can infer from the definition of $W_4$ that
\begin{align*}
\lim_{\alpha\to0}B^3_{n, p, q, r,\alpha}=\frac{1}{n-1}- \frac{n-1}{4}\left(\left(\frac{n+1}{n-1} -\frac{q}{p-1}\right)^2+\left(\frac{n+1}{n-1}-\frac{r}{p-1}\right)^2\right)>0,
\end{align*}
so there exists $\alpha_4\geq 0$ such that there holds $B^3_{n, p, q, r,\alpha}>0$ if $\alpha\geq\alpha_4$.
\end{enumerate}

We denote
		$$
		\alpha_0 :=\alpha_0(n,p,q,r)=
		\begin{cases}
			\alpha_1, & \text{$(b, c, q, r)$ lies in $W_1$} ;\\
			\alpha_2 , &\text{$(b, c,  q, r)$ lies in $W_2\setminus W_1$} ; \\
			\alpha_3, &\text{$(b, c, q, r)$ lies in $W_3\setminus W_1$}; \\
			\alpha_4,  & \text{$(b, c, q, r)$ lies in $W_4\setminus (W_1\cup W_2\cup W_3)$},
		\end{cases}
		$$
and
		$$
		\beta_0 :=\beta_0(n,p,q,r)=
		\begin{cases}
			\frac{2}{n-1}, & \text{$(b, c,  q, r)$ lies in $W_1$} ;\\
			B^1_{n,p,q,r,\alpha_2} ,&\text{$(b, c,  q, r)$ lies in $W_2\setminus W_1$} ; \\
			B^2_{n,p,q,r,\alpha_3} ,  &\text{$(b, c,   q, r)$ lies in $W_3\setminus W_1$}; \\
			B^3_{n,p,q,r,\alpha_4} ,  & \text{$(b, c,  q, r)$ lies in $W_4\setminus (W_1\cup W_2\cup W_3)$}.
		\end{cases}
		$$
Thus we finish the proof.
\end{proof}
	
Since the estimate for $\mL(f^\alpha)$ in \lemref{lem2.5} depends only on the signs of $b$ and $c$, we can obtain directly the following
\begin{cor}\label{21}
If $b>0$, $c>0$ and $(q,r)$ satisfies
		$$
		\frac{q}{p-1}\leq \frac{n+1}{n-1} \quad \text{and}\quad
		\frac{r}{p-1} <\frac{n+3}{n-1},
		$$
or
		$$
		\frac{r}{p-1}\leq \frac{n+1}{n-1} \quad \text{and}\quad
		\frac{q}{p-1} <\frac{n+3}{n-1},
		$$
or
		$$
		\left( \frac{n+1}{n-1}-\frac{q}{p-1}\right)^2
		+\left(\frac{n+1}{n-1}-\frac{r}{p-1}\right)^2
		<\frac{4}{(n-1)^2},
		$$
then the following estimate
\begin{align*}
			\mathcal{L} (f^{\alpha_0 })
			\geq&
			2\alpha_0 \beta_{0} f^{\alpha_0 +\frac{p}{2} }
			-2\alpha_0 (n-1)\kappa  f^{\alpha_0 +\frac{p}{2}-1}
			- \alpha_0  a_1|\nabla f|f^{\alpha_0 +\frac{p-3}{2}}
\end{align*}
holds true on $\{x:f(x)\neq 0\}$.
\end{cor}
	
\begin{proof}
Under the assumption $b>0$ and $c>0$ (so $b_1>0$ and $c_1>0$), it is easy to see that $(b, c, q, r)$ belonging to the union $W_1\cup W_2$, where $W_1$ and $W_2$ are defined in \lemref{lem2.5}, is equivalent to that $ q$ and $r$ fulfill
		$$
		q\leq\frac{n+1}{n-1}(p-1)\quad\mbox{and}\quad  r<\frac{n+3}{n-1}(p-1);
		$$
$(b, c, q, r)$ belonging to the union $W_1\cup W_3$ is equivalent to that $q$ and $r$ satisfy
		$$
		q<\frac{n+3}{n-1}(p-1)\quad\mbox{and}\quad r\leq\frac{n+1}{n-1}(p-1);
		$$
and $(b, c, q, r)$ belonging to $W_4$ is just
		$$
		\left(\frac{n+1}{n-1}-\frac{q}{p-1}\right)^2+\left(\frac{n+1}{n-1}-\frac{r}{p-1}\right)^2<
		\frac{4}{(n-1)^2}.
		$$
\end{proof}

Analogous to the above theorem, we can obtain that
\begin{cor}\label{27}
If $b<0$ and $c<0$ and $(q, r)$ satisfies
		$$
		\frac{q}{p-1}\geq \frac{n+1}{n-1} \quad\quad \text{and}\quad\quad
		\frac{r}{p-1} >1,
		$$
or
		$$
		\frac{r}{p-1}\geq \frac{n+1}{n-1} \quad\quad \text{and}\quad\quad
		\frac{q}{p-1}>1,
		$$
or
		$$
		\left( \frac{n+1}{n-1}-\frac{q}{p-1}\right)^2
		+\left(\frac{n+1}{n-1}-\frac{r}{p-1}\right)^2
		<\frac{4}{(n-1)^2},
		$$
then the following estimate
		\begin{align*}
			\mathcal{L} (f^{\alpha_0 })
			\geq&
			2\alpha_0 \beta_{0} f^{\alpha_0 +\frac{p}{2} }
			-2\alpha_0 (n-1)\kappa  f^{\alpha_0 +\frac{p}{2}-1 }
			- \alpha_0  a_1   |\nabla f|f^{\alpha_0 +\frac{p-3}{2} }
		\end{align*}
holds true on $\{x:f(x)\neq 0\}$.
\end{cor}
	
\section{Gradient estimate for the solutions to (\ref{equ0})}\label{sect3}
In this section, we divide the proof of our main theorem, i.e., Theorem \ref{t1}, into three parts. In the first part, we need to derive a basic integral inequality of $f=|\nabla u|^2$, which will be used in the second and third parts. In the second part, we show how to pick some fixed $\beta>0$ and give an $L^{\beta}$ estimate of $f$ on a ball with radius $3R/4$. This $L^{\beta}$ bound of $f$ will be chosen as the initial integral of the Moser iteration. Finally, we will give a complete proof of our theorem using the Moser iteration method.
	
\subsection{Integral inequality on the solutions to (\ref{equ0}): the case $\dim(M)>2$}
\begin{lem}\label{lem3.2}
Let $(M,g)$ be an $n$-dim complete Riemannian manifold with $\ric_g\geq -(n-1)\kappa g$ where $\kappa$ is a non-negative constant and $\Omega = B_R(o)\subset M$ be a geodesic ball. Assume that $v$ is a positive solution to the equation (\ref{equ0}) with constants $p,\,q$ and $r$ which satisfy the condition in \thmref{t1}(or \lemref{lem2.5}), $u = -(p-1)\ln v$ and $f=|\nabla u|^2$. Then there exist constants $a_3,\, a_4$ and $a_5$ depending only on $n$, $p$, $q$ and $r$ such that for any $t\geq t_0 $, where $t_0$ is defined in \eqref{equa2.8}, there holds
\begin{align*}
\begin{split}
&\beta_0\int_\Omega f^{\alpha_0+\frac{p}{2}+t}\eta^2 + \frac{a_3}{ t }e^{-t_0}V^{\frac{2}{n}}R^{-2}\left\|f^{\frac{\alpha_0+t-1}{2}+\frac{p}{4} }\eta\right\|_{L^{\frac{2n}{n-2}}}^2\\
\leq\  & a_5t_0^2R^{-2} \int_\Omega f^{\alpha_0+\frac{p}{2}+t-1}\eta^2+\frac{a_4}{t }\int_\Omega f^{\alpha_0+\frac{p}{2}+t-1}|\nabla\eta|^2,
\end{split}
\end{align*}
where $\eta\in C^{\infty}_0(\Omega,\bR)$.
\end{lem}
	
\begin{proof}
By the regularity theory on elliptic equations, away from $\{f=0\}$, $u$ is smooth. So, the both sides of \eqref{equa2.20} are in fact smooth. Let $\epsilon>0$ and $\psi = f_\epsilon^{\alpha}\eta^2 $, where $f_\epsilon = (f-\epsilon)^+$, $\eta\in C^{\infty}_0(B_R(o))$ is non-negative, and $\alpha>1$ which will be determined later.

Multiplying the both sides of \eqref{equa2.20} by $\psi$, integrating then over $\Omega$ and taking a direct computation yield
\begin{align}\label{equa:3.17}
\begin{split}
&-\int_\Omega \alpha_0 tf^{\alpha_0+\frac{p}{2}-2}f_{\epsilon}^{t-1}|\nabla f|^2\eta^2 + t\alpha_0(p-2)f^{\alpha_0+\frac{p}{2}-3}f_{\epsilon}^{t-1}\la\nabla f,\nabla u\ra^2\eta^2\\
&-\int_\Omega2\eta\alpha_0 f^{\alpha+\frac{p}{2} -2}f_{\epsilon}^t\la\nabla f,\nabla\eta\ra+2\alpha_0\eta(p-2)f^{\alpha+\frac{p}{2}-3}f_{\epsilon}^t\la\nabla f,\nabla u\ra\la\nabla u,  \nabla\eta\ra\\
\geq & 2\beta_{0}\alpha_0 \int_\Omega f^{\alpha_0+\frac{p}{2}}f_{\epsilon}^t\eta^2  -2(n-1)\alpha_0\kappa\int_\Omega f^{\alpha_0+\frac{p}{2}-1}f_{\epsilon}^t\eta^2 - a_1\alpha_0\int_\Omega f^{\alpha_0+\frac{p-3}{2}}f_{\epsilon}^t|\nabla f|\eta^2.
\end{split}
\end{align}
If we denote $a_2 = \min\{1,\, p-1\}$, then we have
\begin{align}\label{2.24}
f_{\epsilon}^{t-1}|\nabla f|^2 +(p-2)f_{\epsilon}^{t-1}f^{-1}\la\nabla f,\nabla u\ra^2\geq a_2 f_{\epsilon}^{t-1}|\nabla f|^2
\end{align}
and
\begin{align}\label{2.25}
f_{\epsilon}^t\la\nabla f,\nabla\eta\ra+ (p-2)f_{\epsilon}^tf^{-1}\la\nabla f,\nabla u\ra\la\nabla u,  \nabla\eta\ra\geq -(p+1)f_{\epsilon}^t |\nabla f||\nabla\eta|.
\end{align}
Plugging (\ref{2.24}) and (\ref{2.25}) into (\ref{equa:3.17}) and letting $\epsilon\to 0$, we obtain
\begin{align}\label{2.26}
\begin{split}
&\ 2\beta_0\int_\Omega f^{\alpha_0+\frac{p}{2}+t}\eta^2 + \int_\Omega a_2  tf^{\alpha_0+\frac{p}{2}+t-3}|\nabla f|^2\eta^2\\
\leq &~2(n-1) \kappa\int_\Omega f^{\alpha_0+\frac{p}{2}+t-1}\eta^2 + a_1 \int_\Omega f^{\alpha_0+\frac{p-3}{2}+t }|\nabla f|\eta^2\\
& +2 (p-1)\int_\Omega  f^{\alpha_0+\frac{p}{2}+t-2}|\nabla f||\nabla\eta|\eta.
\end{split}
\end{align}
Here we have divided the both sides of the inequality by $\alpha_0$.
		
Using Cauchy-inequality, we have
\begin{align}\label{2.27}
\begin{split}
				a_1 f^{\alpha_0+\frac{p-3}{2}+t }|\nabla f|\eta^2\leq &\
				\frac{a_2t}{4}    f^{\alpha_0+\frac{p}{2}+t-3}|\nabla f|^2\eta^2
				+\frac{ a_1^2}{a_2t} f^{\alpha_0+\frac{p}{2}+t}\eta^2;
				\\
				2(p+1)f^{\alpha_0+\frac{p}{2}+t-2}|\nabla f||\nabla\eta|\eta\leq &\
				\frac{a_2t}{4}    f^{\alpha_0+\frac{p}{2}+t-3}|\nabla f|^2\eta^2
				+\frac{4(p+1)^2 }{a_2t} f^{\alpha_0+\frac{p}{2}+t-1}|\nabla \eta|^2.
\end{split}
\end{align}
Now we choose $t$ large enough such that
		\begin{align}\label{2.28}
			\frac{a_1^2}{a_2t}\leq  \beta_0.
		\end{align}
It follows from (\ref{2.26}), (\ref{2.27}) and (\ref{2.28}) that
		\begin{align}
			\label{2.29}
			\begin{split}
				& \beta_0\int_\Omega f^{\alpha_0+\frac{p}{2}+t}\eta^2
				+
				\frac{a_2  t}{2}\int_\Omega f^{\alpha_0+\frac{p}{2}+t-3}|\nabla f|^2\eta^2 \\
				\leq  &\
				2(n-1) \kappa\int_\Omega f^{\alpha_0+\frac{p}{2}+t-1}\eta^2
				+\frac{4(p+1)^2 }{a_2t}  \int_\Omega f^{\alpha_0+\frac{p}{2}+t-1}|\nabla \eta|^2.
			\end{split}
		\end{align}
		
On the other hand, we have
\begin{align}\label{2.30}
\begin{split}
\left|\nabla \left(f^{\frac{\alpha_0+t-1}{2}+\frac{p}{4} }\eta \right)\right|^2\leq &\ 2\left|\nabla f^{\frac{\alpha_0+t-1}{2}+\frac{p}{4} }\right|^2\eta^2 +2f^{\alpha_0+t-1+\frac{p}{2}}|\nabla\eta|^2\\
=&\ \frac{(\alpha_0+t+\frac{p}{2}-1)^2}{2}f^{\alpha_0+t+\frac{p}{2}-3}|\nabla f |^2\eta^2 +2f^{\alpha_0+t-1+\frac{p}{2}}|\nabla\eta|^2  .
\end{split}
\end{align}
Substituting (\ref{2.30}) into (\ref{2.29}) gives
		\begin{align}
			\label{2.31}
			\begin{split}
				& \beta_0 \int_\Omega f^{\alpha_0+\frac{p}{2}+t}\eta^2
				+
				\frac{4a_2t}{(2\alpha_0+2t+p-2)^2}\int_\Omega   \left|\nabla \left(f^{\frac{\alpha_0+t-1}{2}+\frac{p}{4} }\eta\right)\right|^2 \\
				\leq  &\
				2(n-1)\kappa  \int_\Omega f^{\alpha_0+t+\frac{p}{2}-1}\eta^2
				+
				\frac{4(p-1)^2 }{a_2t} \int_\Omega f^{\alpha_0+\frac{p}{2}+t-1}|\nabla\eta|^2\\
				&\ +
				\frac{8a_2t}{(2\alpha_0+2t+p-2)^2}\int_\Omega f^{\alpha_0+t+\frac{p}{2}-1}|\nabla\eta|^2 .
			\end{split}
		\end{align}
		
From now on we use $a_1,\, a_2,\, a_3\, \cdots$ to denote constants depending only on $n,\, p,\,q$, and $r$. We choose $a_3$ and $a_4$ such that
\begin{align}\label{2.32}
\frac{a_3}{t}\leq  \frac{4a_2t}{(2\alpha_0+2t+p-2)^2}\quad\quad\text{and}\quad\quad \frac{8a_2t}{(2\alpha_0+2t+p-2)^2}+\frac{4(p-1)^2}{a_2t} \leq\frac{a_4}{t},
\end{align}
since $t$ satisfies (\ref{2.28}). It follows from (\ref{2.31}) and (\ref{2.32}) that
\begin{align}\label{2.33}
\begin{split}
&\beta_0\int_\Omega f^{\alpha_0+\frac{p}{2}+t}\eta^2 + \frac{a_3}{t}\int_\Omega   \left|\nabla \left(f^{\frac{\alpha_0+t-1}{2}+\frac{p}{4} }\eta\right)\right|^2 \\
\leq  &\ 2(n-1)\kappa  \int_\Omega f^{\alpha_0+t+\frac{p}{2}-1}\eta^2 + \frac{a_4 }{t} \int_\Omega f^{\alpha_0+\frac{p}{2}+t-1}\left|\nabla\eta\right|^2.
\end{split}
\end{align}
We also note that Saloff-Coste's Sobolev inequality implies
$$e^{-c_0(1+\sqrt{\kappa}R)}V^{\frac{2}{n}}R^{-2}\left\|f^{\frac{\alpha_0+t-1}{2}+\frac{p}{4} }\eta\right\|_{L^{\frac{2n}{n-2}}(\Omega)}^2\leq \int_{\Omega}\left|\nabla \left(f^{\frac{\alpha_0+t-1}{2}+\frac{p}{4} }\eta\right)\right|^2+R^{-2}\int_\Omega f^{ \alpha_0+t +\frac{p}{2} -1 }\eta ^2.$$
Thus, we can infer from the above Saloff-Coste's Sobolev inequality and \eqref{2.33} that
\begin{align}\label{3.32}
\begin{split}
& \beta_0\int_\Omega f^{\alpha_0+\frac{p}{2}+t}\eta^2 + \frac{a_3}{t }e^{-c_0(1+\sqrt{\kappa}R)} V^{\frac{2}{n}}R^{-2}\left\|f^{\frac{\alpha_0+t-1}{2}+\frac{p}{4} }\eta\right\|_{L^{\frac{2n}{n-2}}}^2\\
\leq  &\ 2(n-1)\kappa  \int_\Omega f^{\alpha_0+t+\frac{p}{2}-1}\eta^2+\frac{a_4}{ t }\int_\Omega f^{\alpha_0+t+\frac{p}{2}-1}|\nabla\eta|^2
+\frac{a_3}{ t }\int_\Omega R^{-2}f^{ \alpha_0 +\frac{p}{2}+t-1 }\eta ^2.
\end{split}
\end{align}
		
Now we set
\begin{align}\label{equa2.8}
t_0 = c_{n,p,q,r}\left(1+\sqrt{\kappa} R\right) \quad\text{and}\quad c_{n,p,q,r}=\max\left\{c_0+1, \frac{a_1^2}{a_2\beta_{n,p,q,r}} \right\},
\end{align}
and pick $t$ such that $t\geq t_0$. Since
		\begin{align*}
			2(n-1)\kappa  R^2\leq\frac{ 2(n-1)}{c_{n,p,q,r}^2}t_0^2\quad \text{ and}\quad \frac{a_3}{t}\leq \frac{a_3}{c_{n,p,q,r}},
		\end{align*}
then, there exists $a_5 = a_5(n,p,q,r)>0$ such that
		\begin{align}\label{a5}
			2(n-1)\kappa  R^2+\frac{a_3}{t}\leq a_5t_0^2 = a_5c^2_{n,p,q,r}\left(1+\sqrt{\kappa} R\right)^2.
		\end{align}
Hence, it follows from (\ref{3.32}) and (\ref{a5}) that
\begin{align}\label{2.34}
\begin{split}
&\beta_0\int_\Omega f^{\alpha_0+\frac{p}{2}+t}\eta^2 + \frac{a_3}{ t }e^{-t_0}V^{\frac{2}{n}}R^{-2}\left\|f^{\frac{\alpha_0+t-1}{2}+\frac{p}{4} }\eta\right\|_{L^{\frac{2n}{n-2}}}^2\\
\leq\ & a_5t_0^2R^{-2} \int_\Omega f^{\alpha_0+\frac{p}{2}+t-1}\eta^2+\frac{a_4}{t }\int_\Omega f^{\alpha_0+\frac{p}{2}+t-1}|\nabla\eta|^2.
\end{split}
\end{align}
Thus, we finish the proof.
\end{proof}

\subsection{ \texorpdfstring{$L^{\beta}$}\, bound of gradient in a ball with radius $3R/4$: the case $\dim(M)>2$}\label{sect3.2}
\begin{lem}\label{lpbound}
Let $(M,g)$ be an $n$-dim complete Riemannian manifold with $\ric_g\geq -(n-1)\kappa g$ where $\kappa$ is a non-negative constant and $\Omega = B_R(o)\subset M$ be a geodesic ball. Assume that $f$ is the same as in the above lemma. If $bc\geq 0$ and $(b, c,  q, r)\in W_1\cup W_2\cup W_3\cup W_4$, then, for $\beta = \left(\alpha_0 + t_0 + \frac{p}{2} -1\right)\frac{n}{n-2}$, there exists a positive constant $a_8 = a_8(n,p,q)>0$ such that
\begin{align}\label{lpbpund}
\|f \|_{L^{\beta}(B_{3R/4}(o))}\leq a_8V^{\frac{1}{\beta}} \frac{t_0^2}{ R^2},
\end{align}
where $V$ is the volume of geodesic ball $B_R(o)$.
\end{lem}
	
\begin{proof}
We set $t=t_0$ in \eqref{2.34}, and obtain
\begin{align}\label{equation:3.31}
\begin{split}
& \beta_0\int_\Omega f^{\alpha_0+\frac{p}{2}+t_0}\eta^2 + \frac{a_3}{ t_0 }e^{-t_0}V^{\frac{2}{n}}R^{-2}\left\|f^{\frac{\alpha_0+t_0-1}{2}+\frac{p}{4} }\eta\right\|_{L^{\frac{2n}{n-2}}}^2\\
\leq  &\ a_5t_0^2R^{-2} \int_\Omega f^{\alpha_0+\frac{p}{2}+t_0-1}\eta^2+\frac{a_4}{t_0 }\int_\Omega f^{\alpha_0+\frac{p}{2}+t_0-1}|\nabla\eta|^2.
\end{split}
\end{align}
From (\ref{equation:3.31}) we know that, if
$$
f\geq  \frac{2a_{5}t_0^2}{\beta_0R^2},
$$
then there holds true
$$a_5t_0^2R^{-2} \int_\Omega f^{\alpha+\frac{p}{2}+t_0-1}\eta^2\leq\frac{\beta_{n,p,q,r}}{2}\int_\Omega f^{\alpha_0+\frac{p}{2}+t_0}\eta^2.$$

Now we denote $\Omega_1 = \left\{f<  \frac{2a_{5}t_0^2}{\beta_{n,p,q,r}R^2} \right\}$, then $\Omega=\Omega_1\cup\Omega_2$ where $\Omega_2$ is the complement of $\Omega_1$. Hence, we have
		\begin{align}\label{2.36}
			\begin{split}
				a_5t_0^2R^{-2} \int_\Omega f^{\alpha_0+\frac{p}{2}+t_0-1}\eta^2
				=&\ a_5t_0^2R^{-2} \int_{\Omega_1} f^{\alpha_0+\frac{p}{2}+t_0-1}\eta^2
				+a_5t_0^2R^{-2} \int_{\Omega_2} f^{\alpha_0+\frac{p}{2}+t_0-1}\eta^2
				\\
				\leq&\
				\frac{a_5t_0^2}{R^2} \left(\frac{2a_{5}t_0^2}{\beta_0R^2}\right)^{\alpha_0+\frac{p}{2} +t_0-1 }V
				+
				\frac{\beta_0 }{2}\int_\Omega f^{\alpha_0+\frac{p}{2}+t_0}\eta^2,
			\end{split}
		\end{align}
		where $V$ is the volume of $B_R(o)$.
		It follows from (\ref{2.34}) and (\ref{2.36})
		\begin{align}
			\label{2.37}
			\begin{split}
				&\frac{\beta_0}{2}\int_\Omega f^{\alpha_0+\frac{p}{2}+t_0}\eta^2
				+
				\frac{a_3}{ t_0 }e^{-t_0}V^{\frac{2}{n}}R^{-2}\left\|f^{\frac{\alpha_0+t_0-1}{2}+\frac{p}{4} }\eta\right\|_{L^{\frac{2n}{n-2}}}^2\\
				\leq \ &
				\frac{a_5t_0^2}{R^2} \left(\frac{2a_{5}t_0^2}{\beta_0R^2}\right)^{\alpha_0+\frac{p}{2} +t_0-1 }V
				+\frac{a_4}{ t_0 }\int_\Omega f^{\alpha_0+\frac{p}{2}+t_0-1}|\nabla\eta|^2.
			\end{split}
		\end{align}
		
We choose $\gamma\in C^{\infty}_0(B_R(o))$ such that
		$$
		\begin{cases}
			0\leq\gamma(x)\leq 1,\quad |\nabla\gamma(x)|\leq\frac{C }{R}, &\forall x\in B_R(o);\\
			\gamma(x)\equiv 1, & \forall x\in B_{ {3R}/{4}}(o),
		\end{cases}
		$$
and let $\eta = \gamma^{ \alpha_0 + \frac{p}{2}+t_0}$. Then, we have
		\begin{align}\label{314}
			a_4R^2 |\nabla\eta|^2
			\leq
			a_4 C^2 \left( t_0+ \frac{p}{2}+\alpha_0\right )^2\eta ^{\frac{2\alpha_0+2t_0+p-2}{\alpha_0+p/2+t_0}}
			\leq
			a_{6}t^2_0\eta^{\frac{2\alpha_0+2t_0+p-2}{\alpha_0+p/2+t_0}}.
		\end{align}
By H\"older inequality and Young inequality, we have
\begin{align}\label{2.39}
\begin{split}
\frac{a_4}{t_0}\int_{\Omega}f^{\frac{p}{2}+\alpha_0+t_0-1}|\nabla\eta|^2 \leq\  &\frac{a_6t_0}{R^2} \int_{\Omega}f^{\frac{p}{2}+\alpha_0+t_0-1}
\eta^{\frac{2\alpha_0+p+2t_0-2}{\alpha_0+p/2+t_0}}\\				
\leq\ & \frac{a_6t_0}{R^2}  \left(\int_{\Omega}f^{\alpha_0+t_0+ \frac{p}{2} }\eta^2\right)^{\frac{\alpha_0+p/2+t_0-1}{\alpha+p/2+t_0}}V^{\frac{ 1}{\alpha_0+t_0+ p/2 }}\\
\leq \ &\frac{\beta_0}{2}\left[\int_{\Omega}f^{ \alpha_0+t_0+\frac{p}{2} }\eta^2 + \left(\frac{2a_{6}t_0 }{\beta_0R^2}\right)^{ \alpha_0 +t_0 +\frac{p}{2}}V\right].
\end{split}
\end{align}
Hence, we conclude from (\ref{2.37}) and (\ref{2.39}) that
		\begin{align}
			\label{2.40}
			\begin{split}
				&\left(\int_{\Omega}f^{\frac{n(p/2+\alpha_0+t_0-1)}{n-2}}\eta^{\frac{2n}{n-2}}\right)^{\frac{n-2}{n}}
				\\
				\leq\
				&
				\frac{t_0}{a_3} e^{t_0}V^{1-\frac{2}{n}}R^2
				\left[\frac{2a_5t_0^2}{R^2} \left(\frac{2a_{5}t_0^2}{\beta_0R^2}\right)^{\alpha_0+t_0+\frac{p}{2}-1 }
				+
				\frac{a_{6}t^2_0}{R^2} \left(\frac{2a_{6}t_0 }{\beta_0R^2}\right)^{\alpha_0+t_0+\frac{p}{2} -1 }\right]
				\\
				\leq\
				&
				a_7^{t_0}e^{t_0}V^{1-\frac{2}{n}}t_0^3
				\left( \frac{t_0^2}{ R^2}\right)^{\alpha_0+t_0+\frac{p}{2} -1},
			\end{split}
		\end{align}
where the constant $a_7$ depends only on $n,\,p,\, q, \, r$  and satisfies
		$$
		a_7^{t_0} \geq \frac{2a_5}{a_3}\left(\frac{4a_5}{\beta_0}\right)^{\alpha_0+t_0+\frac{p}{2}-1}
		+\frac{a_6}{a_3}\left(\frac{4a_6}{\beta_0t_0}\right)^{\alpha_0+t_0+\frac{p}{2}-1}.
		$$
Here we have used the fact $t_0\geq1$.
		
Taking $\frac{1}{\alpha_0+t_0+p/2-1}$ power of the both sides of (\ref{2.40}) gives
		\begin{align}
			\left\|f\eta^{\frac{2}{\alpha_0+t_0+p/2-1}}\right\|_{L^{\beta}(\Omega)}\leq a_7^{\frac{t_0}{\alpha_0+t_0+p/2-1}}V^{1/\beta}
			t_0^{\frac{3}{\alpha_0+t_0+p/2-1}}\frac{t_0^2}{ R^2}\leq a_8V^{\frac{1}{\beta}} \frac{t_0^2}{ R^2},
		\end{align}
		where $a_8$ depends only on $n,\ p,\ q,\ r$ and satisfies
		$$
		a_8 \geq a_7^{\frac{t_0}{\alpha_0+t_0+p/2-1}}t_0^{\frac{3}{\alpha_0+t_0+p/2-1}}.
		$$
		Since $\eta\equiv1$ in $B_{3R/4}$, we obtain that
		$$
		\|f \|_{L^{\beta}(B_{3R/4}(o))}\leq a_8V^{\frac{1}{\beta}} \frac{t_0^2}{ R^2}.
		$$
Thus, we obtain the desired estimate.
\end{proof}

\subsection{Moser iteration}\label{sec3.3}\

Now we turn to providing the proof of \thmref{t1}.
\begin{proof}
We discard the first term in (\ref{2.34}) to obtain
\begin{align}\label{2.42}
			\frac{a_3}{ t }e^{-t_0}V^{\frac{2}{n}}R^{-2}\left\|f^{\frac{\alpha_0+t-1}{2}+\frac{p}{4} }\eta\right\|_{L^{\frac{2n}{n-2}}}^2
			\leq
			\frac{a_5\alpha_0^2}{R^2}  \int_\Omega f^{\alpha_0+\frac{p}{2}+t-1}\eta^2+\frac{a_4}{ t }\int_\Omega f^{\alpha_0+\frac{p}{2}+t-1}|\nabla\eta|^2.
\end{align}
For $k=1, 2,\cdots$, let $r_k = \frac{R}{2}+\frac{R}{4^k}$ and $\Omega_k = B_{r_k}(o)$. Choosing a sequence of cut-off functions $\eta_k\in C^{\infty}(\Omega_k)$ which satisfy
\begin{align}
\begin{cases}
0\leq \eta_k(x)\leq1, \quad |\nabla\eta_k(x)|\leq \frac{C4^k}{R}, & \forall x\in B_{r_{k}}(o);\\
\eta_k(x)\equiv 1, &\forall x\in B_{r_{k+1}}(o),
\end{cases}
\end{align}
where $C$ is a constant that does not depend on $(M,g)$ and equation \eqref{equ0}, and substituting $\eta$ in (\ref{2.42}) by $\eta_k$, we arrive at
\begin{align*}
a_3e^{-t_0}V^{\frac{2}{n}} \left\|f^{\frac{\alpha_0+t-1}{2}+\frac{p}{4} }\eta_k\right\|_{L^{\frac{2n}{n-2}}(\Omega_k)}^2
\leq \ &a_5t_0^2t\int_{\Omega_k} f^{\alpha_0+\frac{p}{2}+t-1}\eta_k^2+ a_4R^2\int_{\Omega_k}f^{\alpha_0+\frac{p}{2}+t-1}\left|\nabla\eta_k\right|^2\\
\leq\ & \left(a_5t_0^2t + C^216^k\right)\int_{\Omega_k} f^{\alpha_0+\frac{p}{2}+t-1}.
\end{align*}
Now we choose $\beta_1=\beta$, $\beta_{k+1}=n\beta_k/(n-2), k=1,2, \ldots$, and let $t=t_k$ such that
		$$t_k+\frac{p}{2}+\alpha_0-1=\beta_k,$$
then we have
\begin{align*}
a_3 \left(\int_{\Omega_k}f^{\beta_{k+1}}\eta_k^\frac{2n}{n-2}\right)^{\frac{n-2}{n}}
\leq\ e^{ t_0}V^{-\frac{2}{n}}\left(a_5t_0^2\left(t_0+\alpha_0+\frac{p}{2}-1\right)\left(\frac{n}{n-2}\right)^k + C^216^k\right)\int_{\Omega_k} f^{\beta_k}.
\end{align*}
Since $n/(n-2)<16, ~\forall n>2$, there exists some constant $a_9 $ such that
\begin{align}\label{eq:2.45}
\left(\int_{\Omega_k}f^{\beta_{k+1}}\eta_k^\frac{2n}{n-2}\right)^{\frac{n-2}{n}}
\leq \ a_9t_0^316^k e^{ t_0}V^{-\frac{2}{n}}  \int_{\Omega_k} f^{\beta_k}.
\end{align}
Taking power of $1/\beta_k$ of the both sides of (\ref{eq:2.45}) and using the fact $\eta_k\equiv1\in\Omega_{k+1}$, we obtain
		\begin{align}
			\label{2.46}
			\begin{split}
				\|f\|_{L^{\beta_{k+1}}(\Omega_{k+1})}
				\leq &
				\left(a_9t_0^316^k e^{ t_0}V^{-\frac{2}{n}}\right)^{\frac{1}{\beta_k}} 16 ^{\frac{k}{\beta_k}}\|f\|_{L^{\beta_k}(\Omega_k)}.
			\end{split}
		\end{align}
Noting
		$$
		\sum_{k=1}^{\infty}\frac{1}{\beta_k} =\frac{n}{2\beta_1} \quad\text{and}\quad \sum_{k=1}^{\infty}\frac{k}{\beta_k}=\frac{n^2}{4\beta_1},
		$$
we have
		\begin{align}
			\label{2.47}
			\|f\|_{L^{\infty}(B_{R/2}(o))}
			\leq \
			a_{10}  V^{-\frac{1}{\beta}} \|f\|_{L^{\beta_1}(B_{3R/4}(o))},
		\end{align}
where
		$$
		a_{10} \geq \left(a_9t_0^316^k e^{ t_0}\right)^{\frac{n}{2\beta_1 }} 16 ^{\frac{n^2}{4\beta_1}}.
		$$
By (\ref{lpbpund}), we obtain
		\begin{align}
			\label{3.44}
			\|f\|_{L^{\infty}(B_{R/2}(o))}
			\leq\ a_{11}\frac{(1+\sqrt{\kappa}R)^2}{R^2},
		\end{align}
where $a_{11} = a_{10}a_8c_{n,p,q,r}$. Reminding $f=|\nabla u|^2$ and $u=-(p-1)\log v$, we obtain the desired estimate. Thus, we complete the proof of \thmref{t1}.
\end{proof}

\subsection{The case $\dim(M)=2$}\

In the proof of \thmref{t1} above, we have used Saloff-Coste's Sobolev inequality on the embedding $W^{1,2}(B)\hookrightarrow L^{\frac{2n}{n-2}}(B)$ on a manifold. Since the Sobolev exponent $2^*=2n/(n-2)$ requires $n>2$, we did not consider the case $n=2$ in \thmref{t1}. In this subsection, we will explain briefly that \thmref{t1} can also be established when $n=2$.
	
	When $\dim M=n=2$, we need the special case of  Saloff-Coste's Sobolev theorem, i.e., \lemref{salof}. For any $n'> 2$, there holds
	$$
	\|f\|_{L^{\frac{2n'}{n'-2}}(B)}^2\leq e^{c_0(1+\sqrt{\kappa}R)}V^{-\frac{2}{n'}}R^2\left(\int_B|\nabla f|^2+R^{-2}f^2\right)
	$$for any $f\in C^{\infty}_0(B)$.
	
	For example, if we choose $n'=4$, then we have	
	\begin{align*}
		\left(\int_{\Omega}f^{2\alpha+p}\eta^{4}\right)^{2}
		\leq e^{c_0(1+\sqrt{\kappa} R)}V^{-\frac{1}{2}}\left(R^2\int_{\Omega}\left|\nabla \left(f^{\frac{p}{4}+\frac{\alpha}{2}}\eta\right)\right|^2+\int_{\Omega}f^{\frac{p}{2}+\alpha}\eta^2\right).
	\end{align*}
	By the above inequality and \eqref{equation:3.31}, we can deduce the following integral inequality by almost the same method as in \lemref{lem3.2}.
	\begin{align} \label{equation3.19}
		\begin{split}
			& \beta_0\int_\Omega f^{\alpha_0+\frac{p}{2}+t}\eta^2 +
			\frac{a_3}{ t }e^{-t_0}V^{\frac{1}{2}}R^{-2}\left\|f^{\frac{\alpha_0+t-1}{2}+\frac{p}{4} }\eta\right\|_{L^{4}}^2\\
			\leq\  & a_5t_0^2R^{-2} \int_\Omega f^{\alpha_0+\frac{p}{2}+t-1}\eta^2+\frac{a_4}{t }\int_\Omega f^{\alpha_0+\frac{p}{2}+t-1}|\nabla\eta|^2,
		\end{split}
	\end{align}
	where $\alpha_0$ is the same as  $\alpha_0$ defined in Section 3,  but the constants $a_i, i=3, 4, 5$ may differ from those defined in Section 3.
	
	By repeating the same procedure as in Section \ref{sect3.2}, we can deduce from \eqref{equation3.19} the $L^{\beta}$-bound of $f $ in a geodesic ball with radius $3R/4$
	\begin{align}\label{equation:3.20}
		\|f \|_{L^{\beta}(B_{3R/4}(o))}\leq a_8V^{\frac{1}{\beta}} \frac{t_0^2}{ R^2},,
	\end{align}
	where $\beta=p+2\alpha_0+2t_0-2$.
	
For the Nash-Moser iteration, we set $\beta_l = 2^l(\alpha_0+t_0+p/2-1)$ and $\Omega_l$ by the similar way with that in Section \ref{sec3.3}, and can obtain the following inequality
	\begin{align}\label{equation:3.21}
		\|f\|_{L^{\infty}(B_{\frac{R}{2}}(o))}
		\leq\ &a_{11} V^{-\frac{1}{\beta}}\|f\|_{L^{\beta}(B_{\frac{3R}{4}}(o))}.
\end{align}
Combining \eqref{equation:3.20} and \eqref{equation:3.21}, we finally obtain the Cheng-Yau type gradient estimate. Harnack inequality and Liouville type results follow from the Cheng-Yau type gradient estimate.
	
\subsection{Proof of \thmref{t3}-\thmref{t4}}\

First we give the proof of \thmref{t3}:
\begin{proof}
For any $x\in B_{R/2}(o)\subset M$,  by \thmref{t1}, we have
\begin{align}\label{3.48}
|\nabla \ln v(x)|\leq c(n,p,q)\frac{ 1+\sqrt{\kappa}R }{R }.
\end{align}
Choosing a minimizing geodesic $\gamma(t)$ with arc length parameter connecting $o$ and $x$:
	$$
	\gamma:[0,d]\to M,\quad\gamma(0)=o, \quad \gamma(d)=x.
	$$
where $d=d(x, o)$ is the geodesic distance from $o$ to $x$, we have
\begin{align}
\ln v(x)-\ln v(o)=\int_0^d\frac{d}{dt}\ln v\circ\gamma(t)dt.
\end{align}
Since $|\gamma'|=1$, we have
\begin{align}
\left|\frac{d}{dt}\ln v\circ\gamma(t)\right|\leq |\nabla \ln v||\gamma'(t)| \leq C(n,p,q, r)\frac{1+\sqrt{\kappa}R}{R }.
\end{align}	
Thus it follows from $d\leq R/2$ and the above inequality that
	$$
	e^{-C(n,p,q, r)(1+\sqrt{\kappa}R)/2}\leq v(x)/v(o)\leq e^{C(n,p,q, r)(1+\sqrt{\kappa}R)/2}.
	$$
So for any $x, y\in B_{R/2}(o)$, we have
	$$
	v(x)/v(y)\leq e^{C(n,p,q, r)(1+\sqrt{\kappa}R)}.
	$$
	
Now suppose $v$ is an entire positive solution of equation \eqref{equ0} on $M$. Letting $R\to\infty$ in \eqref{3.48}, we obtain that
	$$
	|\nabla \ln v(x)|\leq C(n,p,q, r)\sqrt{\kappa}, \quad \forall x\in M.
	$$
For any $y\in M$, choose a minimizing geodesic $\gamma(t)$ with arc length parameter connecting $x$ and $y$:
	$$
	\gamma:[0,d]\to M,\quad\gamma(0)=x, \quad \gamma(d)=y,
	$$
where $d=d(x,y)$ is the distance from $x $ to $y$. We have
\begin{align}\label{3.49}
\ln v(y)-\ln v(x)=\int_0^d\frac{d}{dt}\ln v\circ\gamma(t)dt.
\end{align}
Since $|\gamma'(t)|=1$, we have
\begin{align}\label{3.50}
\left|\frac{d}{dt}\ln v\circ\gamma(t)\right|\leq |\nabla \ln v||\gamma'(t)| = C(n,p,q, r)\sqrt{\kappa}.
\end{align}
It follows from (\ref{3.49}) and (\ref{3.50}) that
\begin{align}\label{3.51}
v(y)/v(x)\leq e^{C(n,p,q,r)\sqrt{\kappa}d(x, y)}.
\end{align}
Thus we finish the proof by (\ref{3.51}).
\end{proof}

In order to give the proof of \thmref{t4} we need to establish the following lemma:

\begin{lem}\label{33}
Let $(M,g)$ be a non-compact complete manifold with $\mathrm{Ric}_g\geq0$. Assume that $v>0$ satisfies the following differential inequality with $1<p<n$
\begin{align}\label{339}		
\Delta_p v+v^r\leq 0\quad &\text{in}\,\, M.
\end{align}
If
$$p-1< r < \frac{n(p-1)}{n-p},$$
then there holds true
\begin{align}\label{340}
\int_{B_R(o)}v^{r}\leq C_1(n,p,r)R^{\frac{(n-p)r-(p-1)n}{r-p+1}},
\end{align}
where $B_R(o)\subset M$ is a geodesic ball and $C_1(n,p,r)$ ia a positive constant depending on $n$, $p$ and $r$.
\end{lem}
	
\begin{proof}		
Let $\varphi\in C_0^1(M)$ be a standard non-negative cutoff function that will be specified later. Let $\alpha<0$ be a parameter that will also be chosen later. Multiplying \eqref{equ42} by $\varphi v^\alpha$ and then integrating by parts,  we obtain
\begin{align}
\begin{split}
\int_Mv^{r+\alpha}\varphi\leq &\int_M \alpha v^{\alpha-1}|\nabla v|^p\varphi+v^{\alpha}|\nabla v|^{p-2}\langle\nabla v, \nabla\varphi\rangle\\
\leq &\int_M \alpha v^{\alpha-1}|\nabla v|^p\varphi+v^{\alpha}|\nabla v|^{p-1}|\nabla\varphi|.
\end{split}
\end{align}
Young inequality implies
\begin{align*}
\int_M v^{\alpha}|\nabla v|^{p-1}|\nabla\varphi| \leq \frac{(p-1)\epsilon^{\frac{p}{p-1}}}{p}\int_M|\nabla v|^pv^{\alpha-1}\varphi
+\frac{1}{p\epsilon^{p}}\int_M\varphi^{1-p}v^{\alpha+p-1}|\nabla\varphi|^p.
\end{align*}
If we denote $$\theta_{\epsilon} = |\alpha|- \frac{(p-1)\epsilon^{\frac{p}{p-1}}}{p}>0\quad\quad \text{and}\quad\quad \theta_{\epsilon}' = \frac{1}{p\epsilon^p},$$
then we have
\begin{align*}
\int_Mv^{r + \alpha}\varphi+\theta_{\epsilon}\int_M  v^{\alpha-1}|\nabla v|^p\varphi
\leq  \theta_{\epsilon}'\int_M\varphi^{1-p}v^{\alpha+p-1}|\nabla\varphi|^p.
\end{align*}

Now we set $\frac{1}{s}+\frac{1}{t}=1$. Using Young inequality again yields
\begin{align}\label{343}
\int_Mv^{r+\alpha}\varphi+\theta_{\epsilon}\int_M  v^{\alpha-1}|\nabla v|^p\varphi \leq \theta_{\epsilon}''\int_M v^{(\alpha+p-1)s}\varphi
+\theta_{\epsilon}'''\int_M  |\nabla\varphi|^{pt}\frac{1}{\varphi^{pt-1}}.
\end{align}
If we choose $$s =\frac{r+\alpha}{\alpha+p-1}$$ (here we need require $\alpha>1-p$), then we have
		$$
		\theta''_{\epsilon} = \frac{\delta^{s}}{sp\epsilon^p} \quad\text{and}\quad
		\theta'''_{\epsilon} = \frac{\delta^{-t}}{tp\epsilon^p}.
		$$
We neglect the first term in the right hand side of \eqref{343} and obtain
\begin{align}\label{344}
C_\epsilon\int_Mv^{r+\alpha}\varphi+\theta_{\epsilon}\int_M  v^{\alpha-1}|\nabla v|^p\varphi \leq\theta_{\epsilon}'''\int_M  |\nabla\varphi|^{pt}\frac{1}{\varphi^{pt-1}}.
\end{align}
		
Now, let us estimate the integral $\int_Mv^r\varphi$. Multiplying $\varphi$ on the both sides of \eqref{equ42} and integrating over $M$ yield
\begin{align*}
\int_Mv^{r}\varphi\leq &\int_M   |\nabla v|^{p-1}| \nabla\varphi|\\
= &\int_M  \left(|\nabla v|^{p-1} v^{\frac{(\alpha-1)(p-1)}{p}} \varphi^{\frac{p-1}{p}}\right)\left(v^{\frac{(1-\alpha)(p-1)}{p}}\varphi^{-\frac{p-1}{p}}|\nabla\varphi|\right)\\
			\leq &
			\left(\int_M |\nabla v|^pv^{\alpha-1}\varphi\right)^{\frac{p-1}{p}}
			\left(\int_Mv^{(1-\alpha)(p-1)}\frac{|\nabla \varphi|^p}{\varphi^{p-1}}\right)^\frac{1}{p}\\
			\leq &
			\left(\int_M |\nabla v|^pv^{\alpha-1}\varphi\right)^{\frac{p-1}{p}}
			\left(\int_Mv^{a(1-\alpha)(p-1)}\varphi
			\right)^\frac{1}{pa}
			\left(\int_M \frac{|\nabla \varphi|^{pa'}}{\varphi^{pa'-1}}\right)^\frac{1}{pa'},
		\end{align*}
where the constant $a$ is determined by
$$a(1-\alpha)(p-1)= r+\alpha\quad \text{and}\quad \frac{1}{a}+\frac{1}{a'}=1.$$
It follows from the above inequality and \eqref{344}
\begin{align*}
\int_Mv^{r}\varphi\leq &\left(\int_M |\nabla v|^pv^{\alpha-1}\varphi\right)^{\frac{p-1}{p}}\left(\int_Mv^{r+\alpha}\varphi
\right)^\frac{1}{pa}\left(\int_M \frac{|\nabla \varphi|^{pa'}}{\varphi^{pa'-1}}\right)^\frac{1}{pa'}\\
\leq &\left(\frac{\theta_\epsilon'''}{\theta_\epsilon'}\right)^{\frac{p-1}{p}}\left(\frac{\theta_\epsilon'''}{C_\epsilon}\right)^{\frac{ 1}{pa}}
\left(\int_M  \frac{|\nabla\varphi|^{pt}}{\varphi^{pt-1}}\right)^\beta\left(\int_M \frac{|\nabla \varphi|^{pa'}}{\varphi^{pa'-1}}\right)^\frac{1}{pa'},
\end{align*}
where
$$
\beta = \frac{p-1}{p}+\frac{(1-\alpha)(p-1)}{p(r+\alpha)}.
$$

Now, we choose $\gamma\in C^{\infty}_0(B_{2R}(o))$ such that
$$
\begin{cases}
\gamma(x)\equiv1, & \quad\text{in}\quad B_R(0);\\
1\geq \gamma(x)\geq 0,\,\,\,\, |\nabla \gamma(x)|\leq\frac{C}{R}, &\quad\text{in}\quad B_{2R}(0),
\end{cases}
$$
and let $\varphi = \gamma^l$. It is easy to see that
		$$
		\frac{|\nabla\varphi|^{pt}}{\varphi^{pt-1}}=l^{pt}\gamma^{l-pt}|\nabla \gamma|^{pt}\quad\quad\text{and}\quad\quad
		\frac{|\nabla\varphi|^{pa'}}{\varphi^{pa'-1}}=l^{pa'}\gamma^{l-pa'}|\nabla \gamma|^{pa'}.
		$$
Now we choose $l$ large enough such that $l>pt$ and $l>pa'$. Since $\mathrm{Ric}\geq 0$, the volume comparison theorem implies
	$$
	\int_M  \frac{|\nabla\varphi|^{pt}}{\varphi^{pt-1}}\leq (lC)^{pt}\int_M R^{-pt}\gamma^{l-pt}\leq (lC)^{pt}\int_{B_{2R}(o)}R^{-pt}\leq  2^n(lC)^{pt}R^{n-pt},
	$$
and
$$
\int_M  \frac{|\nabla\varphi|^{pa'}}{\varphi^{pa'-1}}\leq (lC)^{pa'}\int_M R^{-pa'}\gamma^{l-pa'}\leq  2^n(lC)^{pa'}R^{n-pa'}.
$$
Hence, it follows
\begin{align}\label{45}
\int_{B_R(o)}v^{r}\leq\int_Mv^{r}\varphi \leq C_1R^{\sigma},
\end{align}
where
$$
\sigma = (n-pt)\beta+(n-pa')/pa' = \frac{(n-p)r-(p-1)n}{r-p+1}<0.
$$
Thus, we complete the proof of this lemma.
\end{proof}

\begin{thm}\label{thm34} (\thmref{t4}) Let $(M, g)$ be a non-compact complete manifold with with $\mathrm{Ric}_g \geq 0$ and $\Omega\subset M$ be a domain.

(1). If $1<p<n$ and $r$ satisfies
$$p - 1 < r < \min\left\{\frac{n(p-1)}{n-p},\,\, \frac{n + 3}{n-1}(p-1)\right\},$$
then for any positive solution $v$ to the equation
$$\Delta_pv + v^r =0, \quad\mbox{in}\,\, \Omega$$
and any $o\in\Omega$ such that $d(o, \partial\Omega) > 2R$, we have
$$\sup_{B_R(o)}v \leq C(n,p,r)R^{-\frac{p}{r-p+1}}.$$

(2). If  $1<p<n$, $(q,r)\in U_1\cup U_2\cup U_3$ and
$$p-1 < r < \frac{n(p-1)}{n-p},$$
then for any positive solution $v$ to the equation
$$\Delta_pv + |\nabla u|^q + v^r =0, \quad\mbox{in}\,\, \Omega$$
and any $o\in\Omega$ such that $d(o, \partial\Omega) > 2R$, we have
$$\sup_{B_R(o)}v \leq C(n,p,q, r)R^{-\frac{p}{r-p+1}}.$$
\end{thm}

\begin{proof} We first provide the proof of (1) of \thmref{t4}. From the above \lemref{33} we know that if
$$p - 1 < r < \frac{n(p-1)}{n-p},$$
then the following estimate holds
$$\int_{B_R(o)}v^r \leq C(n,p,r)R^{-\frac{pr}{r-p+1}}V(B_{2R}(o)).$$
Since $\mathrm{Ric}_g(M) \geq 0$, by the volume comparison theorem we have
$$V(B_{2R}) \leq 2^n V(B_R).$$
It follows that
$$\frac{1}{V(B_R(o))}\int_{B_R(o)}v^r \leq C(n,p,r)R^{-\frac{pr}{r-p+1}},$$
where the constant $C(n,p,r)$ may differ from the above $C(n, p, r)$. Since
$$\inf_{B_R(o)}v \leq\left(\frac{1}{V(B_R(o))}\int_{B_R(o)}v^r\right)^{\frac{1}{r}}.$$
On the other hand, by Harnack inequality established in \thmref{t3} (also see \cite{he2023gradient}) we know that if
$$r < \frac{n + 3}{n-1}(p-1),$$
there holds
$$\sup_{B_R(o)} v \leq C \inf_{B_R(o)} v.$$
Hence, it follows immediately that
$$\sup_{B_R(o)}v \leq C(n,p,r)R^{-\frac{p}{r-p+1}}.$$

Now, we turn to showing (2). Since $\Delta_pv + |\nabla u|^q + v^r =0$ implies $\Delta_pv + v^r \leq 0$, we can obtain the same integral estimate
$$\frac{1}{V(B_R(o))}\int_{B_R(o)}v^r \leq C(n,p,r)R^{-\frac{pr}{r-p+1}}$$
as above. The proof of (2) goes almost the same as we did for the case (1) except for we need to note the Harnack inequality holds if $(q,r)\in U_1\cup U_2\cup U_3$, so we omit it. Thus we finish the proof of this theorem.
\end{proof}

\begin{proof}[Proof of \corref{cc5}]
Denote $R = d(x,\partial\Omega)$. By \thmref{t4} or \thmref{thm34}, we have
$$
v(x)\leq \max_{y\in B_{\frac{R}{2}}(0)}v(y)\leq  C(n,p,r)R^{-\frac{p}{r-p+1}}=C(n,p,r)d(x,\partial\Omega)^{-\frac{p}{r-p+1}}.
$$
If $v$ is a positive solution of $\Delta_pv + v^r =0$ in a punctured domain $\Omega\setminus\{o\}$, then it is easy to see that there holds true for any $x$ near $o$
$$v(x) \leq C(n,p,r)d(x, o)^{-\frac{p}{r-p+1}}.$$

Obviously, the proof of (2) in \corref{cc5} is the same as (1), so we omit it and thus the proof is finished.
\end{proof}

\subsection{Singularity analysis on $\mathbb R^n$}\

It is well-known that Pol\'a\v{c}ik, Quittner and Souplet have ever proved the following theorem (see Theorem 3.3 in \cite{MR829846}) for the generalized Lane-Emden equation on an Euclidean space by a complete different method from that employed in the present paper:

\begin{thm*} Let $0< p-1< r <p^*-1$ ($p^*=\frac{(p-1)n+p}{n-p}$) and $\Omega$ be an arbitrary domain of $\mathbb{R}^n$, and assume that $u$ be a positive solution to $$\Delta_pv + v^r=0,\quad\mbox{in}\,\, \Omega.$$
Then there exists $C = C(n,p,r) > 0$ (independent of $\Omega$ and $u$) such that there holds
$$v(x) + |\nabla v(x)|^{\frac{p}{r+1}} \leq C d(x,\partial\Omega)^{-\frac{p}{r+1-p}}, \quad x \in \Omega.$$
In particular, if $\Omega = B_R \setminus \{0\}$ for some $R > 0$, then
$$v(x) + |\nabla v(x)|^{\frac{p}{r+1}} \leq C |x|^{-\frac{p}{r+1-p}},\quad 0 < |x| \leq\frac{R}{2}.$$
\end{thm*}

In the case of Euclidean space, from \thmref{t1} and \corref{cc5} we have the following direct corollary:

\begin{cor}\label{c35}
Let $(M, g)$ be an Euclidean space $\mathbb{R}^n$ and $\Omega\subset \mathbb{R}^n$ be a bounded domain containing $0$. Assume the same condition is given as in (2) of the above \thmref{thm34}. If $v$ is a positive solution of $\Delta_pv + |\nabla u|^q + v^r =0$ in $\Omega$, then there exists $C = C(n,p,r) > 0$ (independent of $\Omega$ and $u$) such that there holds
$$v(x) + |x||\nabla v| \leq C d(x,\partial\Omega)^{-\frac{p}{r+1-p}}, \quad x \in \Omega.$$
In particular, if $\Omega = B_R \setminus \{0\}$ for some $R > 0$, then, for any $x$ near $0$ we have
$$v(x) + |x||\nabla v(x)| \leq C(n,p,q,r)|x|^{-\frac{p}{r-p+1}},$$
where $C(n,p,q,r)$ depends only on $n,\,p,\,q$, and $r$.
\end{cor}
\medskip

Next, let's provide the proofs of a series of corollaries.
\begin{proof}[Proof of \corref{cor14}]
Denote $R = d(x,\partial\Omega)$. By \thmref{t1}, we have
$$
|\nabla \ln v(x)|\leq \sup_{y\in B_{\frac{R}{2}}(0)}|\nabla \ln v(y)|\leq C(n,p,q,r)\frac{1}{R}=C(n,p,q,r)d(x,\partial\Omega)^{-1}.
$$
This is the required estimate and thus the proof is finished.
\end{proof}

\begin{proof}[Proof of \corref{cor15}]
Let $y = \frac{R}{|x|}x$, then $B_R(y)\subset\Omega$ and
$$d(tx+(1-t)y, \partial B_R(y))=t|x|+(1-t)R$$
for any $0\leq t\leq 1$. Since $v$ is a solution of \eqref{eq3.56} in $B_{R}(y)$, by \eqref{eq3.57} we have
\begin{align*}
	|\ln v(x)-\ln v(y)|=&\ \left|\int^1_0\frac{d}{dt}\ln v(tx+(1-t)y)dt\right|
	\\
	\leq &\ |x-y|\int^1_0|\nabla \ln v(tx+(1-t)y) |dt\\
	\leq &\  C(n,p,q,r)|x-y|\int^1_0(t|x|+(1-t)R)^{-1}dt
	\\
	=&\  C(n,p,q,r)(\ln R-\ln|x|),
\end{align*}
i.e.
\begin{align*}
	\ln v(x)
	\leq &\  c_{n,p,q}(\ln R-\ln|x|)+\sup_{|y|=R}\ln v(y).
\end{align*}
Therefore, \eqref{eq3.55} follows from the above inequality immediately.
\end{proof}

\begin{proof}[Proof of \corref{cor15*}]
(i). Since $b\geq 0$, $c\geq 0$ and $p=2$, we can see easily that the spherical average $\bar{v}$ of $v$ on $\{x: |x|=r\}$ is superharmonic, if $v$ is a positive solution to
$$\Delta v + b|\nabla v|^q + cv^r=0.$$
Hence there exists some $m \geq 0$ such that in a neighborhood of $0$
$$\bar{v}(r)\leq m r^{2-n}\quad\mbox{for}\,\, n\geq 3\quad\mbox{and}\quad \bar{v}(r)\leq m \log\frac{1}{|x|}\quad\mbox{for}\,\, n=2.$$
On the other hand, $v$ verifies the Harnack inequality obtained in \thmref{t3}
$$\sup_{|x|=r} v(x)\leq K\inf_{|x|=r}v(x)\quad\mbox{for}\,\, r\in (0,\, R],$$
where $K>0$ is some positive constant. Combined with the Harnack inequality, this yields that in a neighborhood of $0$
$$v(x)\leq C |x|^{2-n}\quad\mbox{for}\,\, n\geq 3\quad\mbox{and}\quad v(x)\leq C\log\frac{1}{|x|} \quad\mbox{for}\,\, n=2.$$

(ii). As $b\leq 0$, $c\leq 0$ and $p=2$, we can see easily that the spherical average $\bar{v}$ of $v$ on $\{x: |x|=r\}$ is subharmonic, if $v$ is a positive solution to
$$\Delta v + b|\nabla v|^q + cv^r=0.$$
Hence there exists some $m \geq 0$ such that in a neighborhood of $0$
$$\bar{v}(r)\geq m r^{2-n}\quad\mbox{for}\,\, n\geq 3\quad\mbox{and}\quad \bar{v}(r)\geq m \log\frac{1}{|x|}\quad\mbox{for}\,\, n=2.$$
According to \thmref{t3} we know that $v$ verifies the Harnack inequality
$$\sup_{|x|=r} v(x)\leq K\inf_{|x|=r}v(x)\quad\mbox{for}\,\, r\in (0,\, R],$$
where $K>0$ is some positive constant. Combined with the Harnack inequality, this yields that in a neighborhood of $0$
$$v(x)\geq C |x|^{2-n}\quad\mbox{for}\,\, n\geq 3\quad\mbox{and}\quad v(x)\geq C\log\frac{1}{|x|} \quad\mbox{for}\,\, n=2.$$
Thus, we finish the proof.
\end{proof}

\begin{proof}[Proof of \corref{cor16}]
We denote by $\delta^*$ the maximal $r>0$ such that any boundary point $a\in\partial\Omega$ belongs to  the boundary of a ball $B_{r}(a_i)\subset\overline{\Omega}$ with radius $r$ and simultaneously to the boundary of a ball $B_r(a_s)\subset\overline{\Omega^c}$ with radius $r$. If $x\in\Omega_{\delta^*}$, we denote by $\sigma(x)$ its projection onto $\partial\Omega$ and $\mathbf{n}_{\sigma(x)}$ the outward normal unit vector of $\partial\Omega$ at $\sigma(x)$ and $z^* = \sigma(x)-\delta_*\mathbf{n}_{\sigma(x)}$, then we have
	\begin{align*}
		|\ln v(x)-\ln v(z^*)|\leq  &\left|\int^1_0\frac{d}{dt}\ln v(tx+(1-t)z^*)dt\right| \\
		= &\left|\int^1_0\langle\nabla\ln v(tx+(1-t)z^*), x-z^*\rangle dt\right|
		\\
		\leq &C(n,p,q,r)|d(x)-\delta^*|\int_0^1(td(x,\partial\Omega)+(1-t)\delta^*)^{-1}dt
		\\
		=& C(n,p,q,r)(\ln \delta^*-\ln d(x,\partial\Omega))
	\end{align*}
	If follows that
	\begin{align*}
		\ln v(x)
		\leq &\  C(n,p,q,r)(\ln \delta^*-\ln d(x,\partial\Omega)+\sup_{d(z^*,\partial\Omega)=R}\ln v(z^*)
	\end{align*}
and
\begin{align*}
v(x)\leq \sup_{d(z^*,\partial\Omega)=\delta^*}  v(z^*)\left(\frac{\delta^*}{d(x,\partial\Omega)} \right)^{C(n,p,q,r)}.
\end{align*}
Thus, we finish the proof.
\end{proof}

\section{Liouville Thoerems}
In this section we discuss the Loiuville properties for positive solutions to the equation \eqref{equ0} or some other similar equations.
\begin{proof}[Proof of \corref{c14}]
In this case, we have $\kappa=0$. By \thmref{t1}, we have
\begin{align}\label{3.45}
\frac{|\nabla v(x_0)|}{v(x_0)}\leq \sup_{B(x_0,\frac{R}{2})}\frac{|\nabla v|}{v}\leq \frac{C(n,p,q,r)}{R}, \quad \forall x_0\in M.
\end{align}
Letting $R\to\infty$ in \eqref{3.45}, we obtain
$$
\nabla v(x_0)=0,\quad\forall x_0\in M.
$$
Thus $v$ is a constant.
\end{proof}
	
Liouville theorem is a direct consequence of gradient estimate. In general, gradient estimate needs a more restrictive assumption on $(q, r)$ than Liouville theorem. To our best knowledge, \corref{c14} is the first Liouville type result for equation (\ref{equ0}) on Riemannian manifolds with non-negative Ricci curvature.
	
In \cite[Theorem 15.2]{MR1879326}, Mitidieri and Pohozaev proved that if $1<p<n$ and
\begin{align}
1<\frac{r}{p-1}<\frac{n}{n-p}\quad\text{or}\quad 1<\frac{q}{p-1}<\frac{n}{n-1},
\end{align}
then the problem
	$$
	\begin{cases}
		\Delta_pv+|\nabla v|^q + v^r\leq 0, &x\in\bR^n;\\
		v\geq0,\quad v\not\equiv 0, &x\in\bR^n,
	\end{cases}
	$$
has no solution in $W^{1,p}_{loc}(\bR^n)$. By a similarly method with that adopted by Mitidieri and Pohozaev, \cite[Theorem 15.2]{MR1879326} can be easily generalized to Riemannian manifold with nonnegative Ricci curvature. In order to prove this fact, we first introduce the following lemma, which is also a direct generalization of \cite[Theorem 12.1]{MR1879326}.

\begin{lem}\label{l41}
Let $(M,g)$ be a non-compact complete manifold with $\mathrm{Ric}_g\geq0$. If
$$p-1< r < \frac{n(p-1)}{n-p},$$
then the following differential inequality
\begin{align}\label{equ42}		
\Delta_p v+v^r\leq 0\quad &\text{in}\,\, M
\end{align}
does not admit any positive solution.
\end{lem}

\begin{proof} From \lemref{33} we know that
\begin{align}\label{4.4}
\int_{B_R(o)}v^{r}\leq\int_Mv^{r}\varphi \leq C_1(n, p, r)R^{\sigma},
\end{align}
where
$$
\sigma = \frac{(n-p)r-(p-1)n}{r-p+1}<0.
$$
Hence, letting $R\to\infty$ in \eqref{4.4} yields that $v\equiv 0$, which contradicts with $v>0$.
\end{proof}

Now we recall the following lemma.
\begin{lem}[He-Hu-Wang]\label{l42}
Let $(M,g)$ be complete non-compact manifold with non-negative Ricci curvature. Suppose that $1<p<n$, $r>0$, $q\geq0$, $q+r>p-1$, and $b>0$. If
\begin{align*}
r+\frac{n-1}{n-p}q<\frac{n(p-1)}{n-p},
\end{align*}
then the problem
\begin{align}
\Delta_pv + bv^r|\nabla v|^q\leq 0 &\quad\text{in}\,\, M
\end{align}
admits no entire positive solutions in $M$.
\end{lem}
	
\begin{prop}[\propref{p5} ]\label{p41}
Let $(M,g)$ be a complete non-compact Riemannian manifold with $\mathrm{Ric}\geq 0$. If
$$1<\frac{q}{p-1}<\frac{n}{n-1}\quad\text{or}\quad 1<\frac{r}{p-1}<\frac{n}{n-p},$$
then the porblem
$$\Delta_p v +|\nabla v|^q +v^r\leq 0,\quad x\in M$$
does not admit any entire positive solution.
\end{prop}

\begin{proof}
If $p>1$ and $r$ satisfies
$$p-1< r < \frac{n(p-1)}{n-p},$$
for any positive solution $u$ to the above problem we have
			$$
			\Delta_p v + v^r\leq 0.
			$$
\lemref{l41} implies that $u\equiv0$, which contradict with the condition $u>0$. Now if
$$p-1< q<\frac{n(p-1)}{n-1},$$
then for any $r>p-1$, Young inequality implies
$$0\geq\Delta_p v +v^r+|\nabla v|^q\geq \Delta_p v +v^{\frac{1}{s} r}|\nabla v|^{\frac{1}{t} q},$$
where $s$ and $t$ are positive constants satisfying $1/s+1/t=1$. Since $q<\frac{n(p-1)}{n-1} $, we can always choose $s$ large enough such that
			$$
			\frac{1}{s}r+\frac{1}{t}\frac{n-1}{n-p}q<\frac{n(p-1)}{n-p}.
			$$
Hence, \lemref{l42} implies that $u\equiv 0$.
\end{proof}

Figure \ref{fig:test1} shows different regions of $(\frac{q}{p-1}, \frac{r}{p-1})$ in \corref{c14} and \propref{p5}.
\begin{figure}[h]
		\centering
		\begin{tikzpicture}
			\draw[help lines, color=red!5, dashed] (-0.5,-0.5) grid (4,4);
			\path[fill=green!30](-0.5, 3)--(4,3)--(4,-0.5)--(-0.5, -0.5);
			\path[fill=green!30](-0.5, 2)--(3,2)--(3,4)--(-0.5, 4);
			\fill[fill=green!30](3,3) circle (1);
			\path[pattern=north west lines](-0.5, 2)--(-0.5, 5)--(2, 5)--(2, 6)--(2.5,6)--(2.5, 5)--(7,5)--(7,2)--(2.5,2)--(2.5, -0.5)--(2, -0.5)--(2,2);
			\draw[dotted] (4 ,-0.5) --(4, 1)  node[anchor=north]{$\frac{q}{p-1}= \frac{n+3}{n-1}$} -- (4 ,4) -- (4 ,5.2);
			\draw[dotted] (-1, 4 )--(1, 4) node[anchor=east]{$\frac{r}{p-1}= \frac{n+3}{n-1}$}-- (4, 4 )  -- (6, 4);
			\filldraw[black] (3,3) circle (1pt) node[anchor=south]{$(\frac{n+1}{n-1}, \frac{n+1}{n-1})$};
			\filldraw[black] (2, 2) circle (0.5pt) node[anchor=north]{$(1, 1)$};
			\filldraw[black] (2.5, 5) circle (0.5pt) node[anchor=south west]{$(\frac{n}{n-1}, \frac{n}{n-p})$};
			\draw[->,ultra thick] (-0.5,0)--(7,0) node[right]{$\frac{q}{p-1}$};
			\draw[->,ultra thick] (0,-0.5)--(0,6) node[above]{$\frac{r}{p-1}$};
		\end{tikzpicture}	
\captionof{figure}{The region filled with green color represents the $(\frac{q}{p-1}, \frac{r}{p-1})$ in \corref{c14}; the region filled with north west lines represents the $(\frac{q}{p-1}, \frac{r}{p-1})$ in \propref{p5}(shadow).}
		\label{fig:test1}
\end{figure}
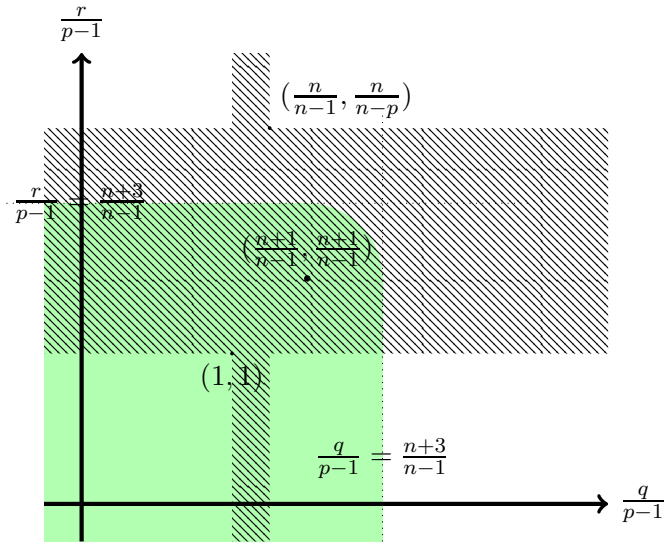
\medskip

\thmref{t6} is the combination of \corref{c14} and \propref{p5}. Now we give the proof of \thmref{t6}.
\begin{proof}
According to the above \corref{c14} and \propref{p5}, for given $p>1$ we first fix several key interval ends on $\frac{q}{p-1}$, i.e.
		$$
		\frac{p}{q-1}=1,\quad\frac{n}{n-1}, \quad \frac{n+1}{n-1},\quad \frac{n+3}{n-1}.
		$$
Then, by the ranges of $\frac{q}{p-1}$ we can determine the ranges of $\frac{r}{p-1}$. Indeed, it is not difficult to see that
\begin{enumerate}
\item if $q$ satisfies $\frac{q}{p-1}\leq 1$, \corref{c14} and \propref{p5} tell us that the corresponding ranges of $r$ are $\frac{r}{p-1}<\frac{n+3}{n-1}$ and $1<\frac{r}{p-1}<\frac{n}{n-p}$ respectively, hence we obtain
		$$\frac{r}{p-1}<\max\left\{\frac{n+3}{n-1},\, \frac{n}{n-p}\right\}.$$
		
\item if $q$ satisfies $$\frac{n}{n-1}>\frac{q}{p-1}>1,$$
by \propref{p5} we know that there is no restriction on $r$.
		
\item if $q$ satisfies $$\frac{n+1}{n-1}\geq\frac{q}{p-1}\geq\frac{n}{n-1},$$
by \corref{c14} and \propref{p5} we know that the corresponding ranges of $r$ are
$$\frac{r}{p-1}<\frac{n+3}{n-1}\quad \mbox{and}\quad 1<\frac{r}{p-1}<\frac{n}{n-p}$$ respectively, hence it follows
$$\frac{r}{p-1}<\max\left\{\frac{n+3}{n-1},\, \frac{n}{n-p}\right\}.$$
		
\item if $q$ satisfies $$\frac{n+3}{n-1}>\frac{q}{p-1}>\frac{n+1}{n-1},$$
by \corref{c14} and \propref{p5} we know that the corresponding ranges of $r$ are $$\frac{r}{p-1}<\frac{n+1}{n-1}+\sqrt{\frac{4}{(n-1)^2}-\left(\frac{n+1}{n-1}-\frac{q}{p-1}\right)^2}\quad \mbox{and} \quad 1<\frac{r}{p-1}<\frac{n}{n-p}$$ respectively,
hence it follows
$$ \frac{r}{p-1}<\max\left\{\frac{n+1}{n-1}+\sqrt{\frac{4}{(n-1)^2}-\left(\frac{n+1}{n-1}-\frac{q}{p-1}\right)^2},\,\, \frac{n}{n-p}\right\}.$$
		
\item if $q$ satisfies $$\frac{q}{p-1}\geq \frac{n+3}{n-1},$$
by \propref{p5} we know the corresponding range of $r$ is
		$$
		1<\frac{r}{p-1}<\frac{n}{n-p}.
		$$
\end{enumerate}
Thus, we complete the proof of the theorem.
\end{proof}

\section{Global gradient estimate}
\thmref{t1} tells us that, if $v$ is an entire solution to \eqref{equ0} on a non-compact complete Riemannian manifold $(M, g)$ with $\mathrm{Ric}_g\geq -(n-1)\kappa$, then there holds true $$\frac{|\nabla v|}{v}\leq C(n,p,q, r)\sqrt{\kappa}.$$
Motivated by Sung and Wang \cite{MR3275651}, we will adopt a similar method with that in \cite{MR3275651} to give the upper bounds of the constants $C(n,p,q,r)$ in different regions $W_i$ ($i=1,\,2,\, 3,\, 4$) in this section.

\subsection{Proof of \thmref{thm1.10}}\

First, we need to establish some auxiliary lemmas.
\begin{lem}\label{lem41}
Let $(M,g)$ be an $n$-dim($n\geq 2$) complete Riemannian manifold satisfying $\mathrm{Ric}_g\geq-(n-1)\kappa g$ for some constant $\kappa\geq0$. Assume $v$ be an entire positive solution of equation \eqref{equ0}. Let $u=-(p-1)\ln v$ and $f = |\nabla u|^2$.
\begin{enumerate}

\item In the case $(b, c, p, q)\in W_1$, we denote $y_1 = (n-1)^2\kappa$ and $\omega = (f-y_1-\delta)^+$, then, on $\Omega_1 = \{f\geq y_1+\delta\}$
there holds true	
\begin{align*}
\mL(\omega^\alpha)\geq 	2\alpha\omega^{\alpha-1}  \left(b_{1}\omega-b_{2}|\nabla \omega|\right), \quad\forall \alpha >2
\end{align*}
where $b_{1}$ and $b_{2}$ are two positive constants.

\item In the case $(b, c, p, q)\in W_2$, we denote
	$$	
	y_2 = \frac{4\kappa}{\left(\frac{n+3}{n-1}-\frac{r}{p-1}\right)\left( \frac{r}{p-1}-1\right)}
	$$
and $\omega = (f-y_2-\delta)^+$. Then, on $\Omega_2 = \{f\geq y_2+\delta\}$ there holds true
\begin{align*}
\mL(\omega^\alpha)\geq 	2\alpha\omega^{\alpha-1}  \left(b_{1}'\omega-b_{2}'|\nabla \omega|\right),
\end{align*}
where $b_{1}'$ and $b_{2}'$ are two positive constants and $\alpha>0$ large enough.
	
\item In the case $(b, c, p, q)\in W_3$, we denote
	$$	
	y_3 = \frac{4\kappa}{\left(\frac{n+3}{n-1}-\frac{q}{p-1}\right)\left( \frac{q}{p-1}-1\right)}
	$$
and $\omega = (f - y_3 -\delta)^+$. Then, on $\Omega_3 = \{f\geq y_3+\delta\}$ there holds true	
\begin{align*}
\mL(\omega^\alpha)\geq 	2\alpha\omega^{\alpha-1}  \left(b_{1}''\omega-b_{2}''|\nabla \omega|\right),
\end{align*}
where $b_{1}''$ and $b_{2}''$ are two positive constants and $\alpha>0$ large enough.

\item In the case $(b, c,p,q)\in W_4$, we denote
	$$	
	y_4 = \frac{\kappa}{\frac{4}{(n-1)^2} - \left(\frac{n+1}{n-1}-\frac{q}{p-1}\right)^2-\left(\frac{n+1}{n-1}-\frac{r}{p-1}\right)^2}
	$$
and $\omega = (f-y_4-\delta)^+$. Then, on $\Omega_4 = \{f\geq y_4+\delta\}$ there holds true	
\begin{align*}
\mL(\omega^\alpha)\geq 	2\alpha\omega^{\alpha-1}  \left(b_{1}'''\omega-b_{2}'''|\nabla \omega|\right),
\end{align*}
where $b_{1}'''$ and $b_{2}'''$ are two positive constants and $\alpha>0$ large enough.
\end{enumerate}
\end{lem}

\begin{proof}
Assume that $(b, c, p, q)\in W_1\cup W_2\cup W_3\cup W_4$. Then, by letting $R\to\infty$ in \eqref{gradient} in \thmref{t1} we obtain
\begin{align*}
	|\nabla u|\leq c(n,p,q,r)\sqrt{\kappa},
\end{align*}
where $c(n,p,q,r)$ may vary in different $W_i$ where $i=1, 2, 3, 4$.

If $(b, c, p, q)\in W_1$, we denote that $y_1 = (n-1)^2\kappa$. Then, for any $\delta>0$ we define $\omega = (f-y_1-\delta)^+$. Hence, on $\{f\geq y_1+\delta\}$ we have
\begin{align*}
\mL(\omega^\alpha) = \di(\alpha\omega^{\alpha-1}A(\nabla f))=\alpha(\alpha-1)\omega^{\alpha-2}\langle\nabla\omega, A(\nabla f)\rangle+\alpha\omega^{\alpha-1}\mL(f).
\end{align*}
In the above expression, it is easy to see that $\nabla\omega=\nabla f$ in $\{f>y_1+\delta\}$, but $\nabla\omega$ causes a distribution on $\{f=y_1+\delta\}$. Now, if we assume $\alpha>2$, then the distribution caused by $\nabla\omega$ on $\{f=y_1+\delta\}$ is eliminated by $\omega$ since $\omega=0$ on $\{f=y_1+\delta\}$. Hence we have
\begin{align*}
\mL(\omega^\alpha) =&\ \alpha(\alpha-1)\omega^{\alpha-2}\langle\nabla f, A(\nabla f)\rangle+\alpha\omega^{\alpha-1}\mL(f)\\
		\geq &\ \omega^{\alpha-1}\left(\alpha(\alpha-1)f^{-1}\langle\nabla f, A(\nabla f)\rangle+\alpha \mL(f)\right)\\
		=&\ \frac{\omega^{\alpha-1}}{f^{\alpha-1}}\mL(f^\alpha).
\end{align*}
It follows from the above inequality and \eqref{w1}
\begin{align*}
\mathcal{L} (f^{\alpha})\geq &\ 2\alpha\omega^{\alpha-1}f^{\frac{p}{2}-1}  \left(\frac{f^2}{n-1}-(n-1)\kappa f-\frac{a_1}{2}f^{\frac{1}{2}}|\nabla \omega|\right)\\
\geq &\ 2\alpha\omega^{\alpha-1}f^{\frac{p-1}{2}}  \left(f^{\frac{1}{2}}\left(\frac{f}{n-1}-(n-1)\kappa \right)-\frac{a_1}{2}|\nabla \omega|\right)\\
\geq& \ 2\alpha\omega^{\alpha-1}f^{\frac{p-1}{2}}  \left(f^{\frac{1}{2}} \frac{\delta}{n-1} -\frac{a_1}{2}|\nabla \omega|\right).
\end{align*}
Using the fact $$y_1+\delta\leq f\leq c(n,p,q,r)\sqrt{\kappa}$$ on $\{f\geq y_1+\delta\}$, we have
\begin{align*}
\mL(f^\alpha)\geq 	2\alpha\omega^{\alpha-1}  \left(b_1\omega-b_2|\nabla \omega|\right), \quad\forall \alpha >2,
\end{align*}
where $b_1$ and $b_2$ are two positive constants.
\medskip
	
If $(b, c, p, q)\in W_2$, by \eqref{w2} we have
\begin{align*}
\frac{f^{2-\alpha-\frac{p}{2}}}{2\alpha}\mathcal{L} (f^{\alpha})\geq &\ B^1_{n, p, q, r,\alpha}f^2-(n-1)\kappa f -\frac{a_1}{2}f^{\frac{1}{2}}|\nabla f|,
\end{align*}
where
	$$
	B^1_{n, p, q, r,\alpha}:=\frac{1}{n-1}-\frac{\left(\frac{n+1}{n-1}-\frac{r}{p-1}\right)^2}{\frac{4}{n-1} -\frac{4d_2}{2\alpha-1} }>0
	$$
and $\alpha$ is large enough such that $B^1_{n, p, q, r,\alpha}>0$.

Denote
\begin{align*}
y_2 = \frac{4\kappa}{\left(\frac{n+3}{n-1}-\frac{r}{p-1}\right)\left( \frac{r}{p-1}-1\right)},
\end{align*}
and let's define $$\omega = (f-y_2-\delta)^+$$
where $\delta>0$ is any positive real number. Then, by taking a computation we have on $\{f\geq y_2+\delta\}$
\begin{align*}
\mL(\omega^\alpha) \geq & \ 2\alpha\omega^{\alpha-1}f^{\frac{p-1}{2}}  \left(f^{\frac{1}{2}}\left(B^1_{n, p, q, r,\alpha} f  -(n-1)\kappa \right)-\frac{a_1}{2}|\nabla \omega|\right)\\
\geq &\ 2\alpha\omega^{\alpha-1}f^{\frac{p-1}{2}}  \left(f^{\frac{1}{2}}\frac{\delta}{2(n-1)}-\frac{a_1}{2}|\nabla \omega|\right)\\
\geq &\	2\alpha\omega^{\alpha-1}  \left(b_1'\omega-b_2'|\nabla \omega|\right).
\end{align*}
\medskip

If $(b, c,p,q)\in W_3$, we denote
	$$
	y_3 = \frac{4\kappa}{\left(\frac{n+3}{n-1}-\frac{q}{p-1}\right)\left( \frac{q}{p-1}-1\right)}
	$$
and define $\omega = (f-y_3-\delta)^+$ where $\delta>0$ is any positive real number. Then, by a similar argument with the second case, on $\{f\geq y_3+\delta\}$ we have
\begin{align*}
\mL(\omega^\alpha) \geq &	2\alpha\omega^{\alpha-1}\left(b_1\omega-b_2|\nabla \omega|\right)
\end{align*}
for large enough $\alpha>0$.
\medskip
	
If $(b, c, p, q)\in W_4$, for $\alpha$ large enough, we have
\begin{align*}
\frac{f^{2-\alpha-\frac{p}{2}}}{2\alpha}\mathcal{L} (f^{\alpha}) \geq &\ B^3_{n, p, q, r,\alpha}f^2-(n-1)\kappa f
-\frac{a_1}{2}f^{\frac{1}{2}}|\nabla f|,
\end{align*}
where
	 $$
	 B^3_{n, p, q, r,\alpha}:=\frac{1}{n-1}-\frac{\left(\frac{n+1}{n-1}-\frac{q}{p-1}\right)^2}{\frac{4}{n-1} -\frac{4d_1}{2\alpha-1} }
	 -\frac{\left(\frac{n+1}{n-1}-\frac{r}{p-1}\right)^2}{\frac{4}{n-1} -\frac{4d_2}{2\alpha-1}}.
	 $$
Now, we denote
\begin{align*}
y_4 = \frac{\kappa}{\frac{4}{(n-1)^2} - \left(\frac{n+1}{n-1}-\frac{q}{p-1}\right)^2-\left(\frac{n+1}{n-1}-\frac{r}{p-1}\right)^2}.
\end{align*}
Let $\omega = (f-y_4-\delta)^+$ for any $\delta>0$. Then, on $\{f\geq y_2+\delta\}$, we have
\begin{align*}
\mL(\omega^\alpha) \geq &\ 2\alpha\omega^{\alpha-1}f^{\frac{p-1}{2}}  \left(f^{\frac{1}{2}}\left(B^3_{n, p, q, r,\alpha} f -(n-1)\kappa \right)-\frac{a_1}{2}|\nabla \omega|\right)\\
\geq &\ 2\alpha\omega^{\alpha-1}f^{\frac{p-1}{2}}  \left(f^{\frac{1}{2}}\frac{\delta}{2(n-1)}-\frac{a_1}{2}|\nabla \omega|\right)\\
\geq &\	2\alpha\omega^{\alpha-1}  \left(b_1'''\omega-b_2'''|\nabla \omega|\right).
\end{align*}
Thus, we complete the proof of this lemma.
\end{proof}

\begin{lem}\label{lem42}
Let $(M,g)$ be an $n$-dim($n\geq 2$) complete Riemannian manifold satisfying $\mathrm{Ric}_g\geq-(n-1)\kappa g$ for some constant $\kappa\geq0$. Let $v\in C^1(M)$ be an entire positive solution of \eqref{equ0}. For some $y>0$, we define $\omega = (f-y)^+$. If $\omega$ satisfies the following inequality
$$	 		
\mL(\omega^\alpha)\geq \omega^{\alpha-1}(b_1\omega-b_2|\nabla\omega|)
$$
where $b_1$, $b_2$ and $\alpha$ are some positive constants, then $\omega\equiv 0$, i.e.,  $f\leq y$.
\end{lem}
	
\begin{proof}
Let $\eta\in C^{\infty}_0(M,\bR)$ be a cut-off function to be determine later. We choose $\eta^2\omega^\gamma$ as test function, then we have
\begin{align*}
\int_M\mL(\omega^\alpha)\omega^{\gamma}\eta^2\geq \int_M2\alpha\omega^{\alpha+\gamma-1}  \left(b_1\omega-b_2|\nabla \omega|\right)\eta^2.
\end{align*}
	 	
We omit the term $f^{\frac{p}{2}-1}$ since $f$ is uniform bounded on $M$. By integration by parts, we obtain
	 	\begin{align*}
	 	&\int_M2\alpha\omega^{\alpha+\gamma-1}  \left(b_1\omega-b_2|\nabla \omega|\right)\eta^2\\
	 	\leq& -\int_M\alpha\gamma\omega^{\alpha+\gamma-2}\eta^2(|\nabla\omega|^2+(p-2)f^{-1}\langle\nabla\omega, \nabla u\rangle^2)\\
	 	&-\int_M2\alpha\eta\omega^{\alpha+\gamma-1}\left(\langle\nabla\omega, \nabla \eta \rangle+(p-2)f^{-1}\langle\nabla\omega,\nabla\eta\rangle\langle \nabla u,\nabla\eta\rangle\right).
	 	\end{align*}
It follows
\begin{align*}
&\int_M2 \omega^{\alpha+\gamma-1}  \left(b_1\omega-b_2|\nabla \omega|\right)\eta^2+a_1\int_M \gamma\omega^{\alpha+\gamma-2}\eta^2|\nabla\omega|^2\\
\leq &2(p+1)\int_M\omega^{\alpha+\gamma-1}|\nabla\omega||\nabla \eta|  \eta.
\end{align*}
Cauchy inequality implies
\begin{align*}
2b_2 \omega^{\alpha+\gamma-1}  |\nabla \omega| \eta^2\leq&\  b_2\eta^2\omega^{\alpha+\gamma-2}\left(\frac{|\nabla\omega|^2}{\epsilon}+\epsilon\omega^2\right) ;\\
2 \eta(p+1) \omega^{\alpha+\gamma-1}|\nabla\omega||\nabla \eta|
\leq& \ (p+1)  \omega^{\alpha+\gamma-2}\left(\frac{\eta^2|\nabla\omega|^2}{\epsilon}+\epsilon|\nabla\eta|^2\omega^2\right).
\end{align*}
If we choose $\epsilon$ such that
	 	\begin{align*}
	 	\frac{b_2+p+1}{\epsilon}=a_1\gamma,
	 	\end{align*}
then we have
	 	\begin{align*}
	 		&\int_M2b_1 \omega^{\alpha+\gamma}   \eta^2
	 	\leq
	 	\int_M  \epsilon(p+1)  \omega^{\alpha+\gamma }|\nabla\eta|^2 +\int_M\epsilon b_2 \omega^{\alpha+\gamma } \eta^2 .
	 	\end{align*}
Picking $\gamma$ large enough such that $\epsilon b_2<b_1$, we obtain
\begin{align*}
&b_1 \int_M \omega^{\alpha+\gamma}\eta^2 \leq \int_M \epsilon(p+1)\omega^{\alpha+\gamma }|\nabla\eta|^2.
\end{align*}
Now, if $\omega\neq 0$, without loss of generality we may assume that $\omega|_{B_1}\neq 0$ for some geodesic ball $B_1$ with radius 1, it follows that there always holds true $$\int_{B_1}\omega^{\alpha+\gamma}>0.$$
By choosing $\eta\in C^{\infty}_0(B_{k+1})$, $\eta\equiv 1$ in $B_k$ and $|\nabla\eta|<4$, then from the above argument we have
\begin{align*}
b_1\int_{B_k}\omega^{\alpha+\gamma}\leq 4\epsilon(p+1) \int_{B_{k+1}}\omega^{\alpha+\gamma}.
\end{align*}
By iteration on $k$, we have
\begin{align*}
\int_{B_k}\omega^{\alpha+\gamma}\geq \left(\frac{C}{\epsilon}\right)^k\int_{B_1}\omega^{\alpha+\gamma}.
\end{align*}
Since $\omega$ is uniformly bounded and $\vol(B_{k})\leq e^{(n-1)k\sqrt{\kappa}}$ by volume comparison theorem, we have
	 	\begin{align*}
	 		C_1^{\alpha+\gamma}e^{(n-1)k\sqrt{\kappa}}\geq e^{k\ln\frac{C}{\epsilon}}.
	 	\end{align*}
We choose $\gamma$ such that $$\ln\frac{C}{\epsilon}>2(n-1)\sqrt{\kappa}+2$$ and then choose $k$ large enough, immediately we get a contradiction. This means $\omega\equiv 0$ and we finish the proof of this lemma.
\end{proof}

\begin{proof}[Proof of \thmref{thm1.10}]
	\thmref{thm1.10} follows immediately from \lemref{lem41} and \lemref{lem42}.
\end{proof}

Note $y_4<y_2<y_1$, $y_4<y_3<y_1$ and $W_i\cap W_j\neq\emptyset$, $\forall 1\leq i,\,j\leq 4$. So we can obtain finer bounds for the intersection parts by choosing the smaller relatively upper bounds. Hence, we can get the following more obvious corollaries.

\begin{cor}
Let $(M,g)$ be an $n$-dim($n\geq 2$) complete Riemannian manifold satisfying $\mathrm{Ric}_g\geq-(n-1)\kappa g$ for some constant $\kappa\geq0$. Let $v\in C^1(M)$ be an entire positive solution of \eqref{equ0} with $b>0$ and $c>0$.
\begin{enumerate}
\item If $p>1$, $q$ and $r$ fulfill the following
$$\frac{q}{p-1}<\frac{n+1}{n-1}\quad\mbox{and}\quad \frac{r}{p-1}<\frac{n+1}{n-1},$$
then we have
				$$
				\frac{|\nabla v|}{v}\leq \frac{(n-1)\sqrt{\kappa}}{p-1}.
				$$
\item If $p>1$, $q$ and $r$ fulfill that
 $$\frac{q}{p-1}<\frac{n+1}{n-1}\quad\mbox{and}\quad \frac{n+1}{n-1}\leq \frac{r}{p-1}<\frac{n+3}{n-1},$$
then we have
$$
\frac{|\nabla v|}{v}\leq \frac{2\sqrt{\kappa}}{(p-1)\sqrt{\left(\frac{n+3}{n-1}-\frac{r}{p-1}\right)\left(\frac{r}{p-1}-1\right)}}.
$$
\item If $p>1$, $q$ and $r$ fulfill that $$\frac{r}{p-1}<\frac{n+1}{n-1}\quad\mbox{and}\quad \frac{n+1}{n-1}\leq \frac{q}{p-1}<\frac{n+3}{n-1},$$
then we have
$$
\frac{|\nabla v|}{v}\leq \frac{2\sqrt{\kappa}}{(p-1)\sqrt{\left(\frac{n+3}{n-1}-\frac{q}{p-1}\right)\left(\frac{q}{p-1}-1\right)}}.
$$
\item If $p>1$, $q$ and $r$ fulfill the following
$$\left( \frac{n+1}{n-1}-\frac{q}{p-1}\right)^2 +\left(\frac{n+1}{n-1}-\frac{r}{p-1}\right)^2 <\frac{4}{(n-1)^2},\quad \frac{q}{p-1}\geq\frac{n+1}{n-1}$$
and $$\frac{r}{p-1}\geq\frac{n+1}{n-1},$$
then we have
$$
\frac{|\nabla v|}{v}\leq \frac{2\sqrt{\kappa}}{(p-1)\sqrt{\frac{4}{(n-1)^2}-\left( \frac{n+1}{n-1}-\frac{q}{p-1}\right)^2 -\left(\frac{n+1}{n-1}-\frac{r}{p-1}\right)^2}}.$$
\end{enumerate}
Figure \ref{fig5} shows the above four regions.
\end{cor}
	
\begin{figure}
\begin{tikzpicture}
			\fill[gray](3,3) circle (1);
			\draw[help lines, color=red!5, dashed] (-0.5,-0.5) grid (4,4);
			\path[fill=red](-0.5, 3)--(3,3)--(3,-0.5)--(-0.5, -0.5);
			\path[fill=green](-0.5, 4)--(3,4)--(3,3)--(-0.5, 3);
			\path[fill=yellow](3, -0.5)--(3,3)--(4,3)--(4, -0.5);
			\draw[dotted] (4 ,-0.5) --(4, 2)  node[anchor=north]{$\frac{q}{p-1}= \frac{n+3}{n-1}$} -- (4 ,4) -- (4 ,4.5);
			\draw[dotted] (3,-0.5) --(3, 1)  node[anchor=north]{$\frac{q}{p-1}= \frac{n+1}{n-1}$} -- (3,4) -- (3,4.5);
			\draw[dotted] (-0.5, 4 )--(2, 4) node[anchor=east]{$\frac{r}{p-1}=\frac{n+3}{n-1}$}-- (4, 4 )  -- (5, 4);
			\draw[dotted] (-0.5,3) --(1,3)  node[anchor=north]{$\frac{r}{p-1}= \frac{n+1}{n-1}$} -- (4,3) -- (4.5,3);
			\draw[->,ultra thick] (-0.5,0)--(5,0) node[right]{$\frac{q}{p-1}$};
			\draw[->,ultra thick] (0,-0.5)--(0,5) node[above]{$\frac{r}{p-1}$};
\end{tikzpicture}			
\begin{tikzpicture}
			\path[fill=green!0](2, 2)--(3,2)--(3,2.5)--(2,2.5);
			\path[fill=red](2, 3)--(3,3)--(3,3.2) node[anchor=west]{$f=|\nabla u|^2<y_1$}--(3,3.5)--(2,3.5);				
			\path[fill=green](2, 4)--(3,4)--(3,4.2) node[anchor=west]{$f=|\nabla u|^2<y_2$}--(3,4.5)--(2,4.5);
			\path[fill=yellow](2, 5)--(3,5)--(3,5.2) node[anchor=west]{$f=|\nabla u|^2<y_3$}--(3,5.5)--(2,5.5);
			\path[fill = gray](2, 6)--(3,6)--(3,6.2) node[anchor=west]{$f=|\nabla u|^2<y_4$}--(3,6.5)--(2,6.5);
\end{tikzpicture}
\caption{When $b>0, c>0$, the region of $(\frac{q}{p-1}, \frac{r}{p-1})$ in $W_1, W_2, W_3, W_4$.}
\label{fig5}
\end{figure}
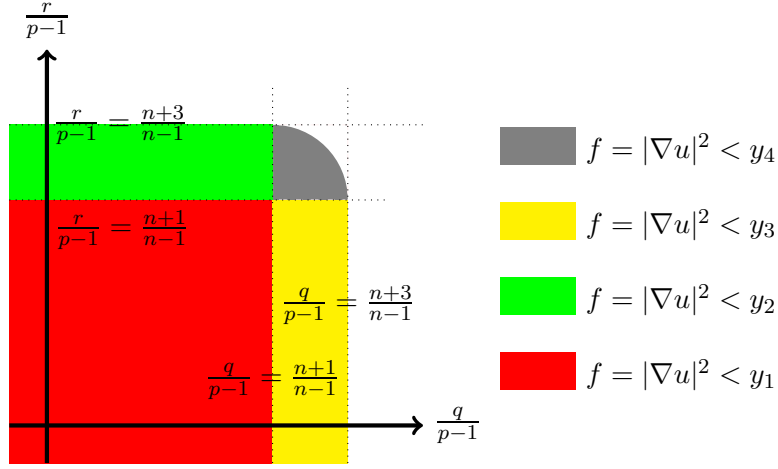

\begin{cor}
Let $(M,g)$ be an $n$-dim($n\geq 2$) complete Riemannian manifold satisfying $\mathrm{Ric}_g\geq-(n-1)\kappa g$ for some constant $\kappa\geq0$. Let $u\in C^1(M)$ be a global positive solution of \eqref{equ0} with $b<0$ and $c<0$.
	
\begin{enumerate}
\item If $p>1$, $q$ and $r$ fulfill the following
	$$\frac{q}{p-1}>\frac{n+1}{n-1}\quad\mbox{and}\quad \frac{r}{p-1}>\frac{n+1}{n-1},$$
	then we have
	$$
	\frac{|\nabla v|}{v}\leq \frac{(n-1)\sqrt{\kappa}}{p-1}.
	$$
\item If $p>1$, $q$ and $r$ fulfill that
	$$\frac{q}{p-1}>\frac{n+1}{n-1}\quad\mbox{and}\quad 1< \frac{r}{p-1}\leq\frac{n+1}{n-1},$$
	then we have
	$$
	\frac{|\nabla v|}{v}\leq \frac{2\sqrt{\kappa}}{(p-1)\sqrt{\left(\frac{n+3}{n-1}-\frac{r}{p-1}\right)\left(\frac{r}{p-1}-1\right)}}.
	$$
\item If $p>1$, $q$ and $r$ fulfill that $$\frac{r}{p-1}>\frac{n+1}{n-1}\quad\mbox{and}\quad 1< \frac{q}{p-1}\leq\frac{n+1}{n-1},$$
	then we have
	$$
	\frac{|\nabla v|}{v}\leq \frac{2\sqrt{\kappa}}{(p-1)\sqrt{\left(\frac{n+3}{n-1}-\frac{q}{p-1}\right)\left(\frac{q}{p-1}-1\right)}}.
	$$
\item If $p>1$, $q$ and $r$ fulfill the following
	$$\left( \frac{n+1}{n-1}-\frac{q}{p-1}\right)^2 +\left(\frac{n+1}{n-1}-\frac{r}{p-1}\right)^2 <\frac{4}{(n-1)^2},\quad \frac{q}{p-1}\leq\frac{n+1}{n-1}$$
	and $$\frac{r}{p-1}\leq\frac{n+1}{n-1},$$
	then we have
	$$
	\frac{|\nabla v|}{v}\leq \frac{2\sqrt{\kappa}}{(p-1)\sqrt{\frac{4}{(n-1)^2}-\left( \frac{n+1}{n-1}-\frac{q}{p-1}\right)^2 -\left(\frac{n+1}{n-1}-\frac{r}{p-1}\right)^2}}.$$
	\end{enumerate}
Figure \ref{fig6} shows the region in the above four regions.
\end{cor}

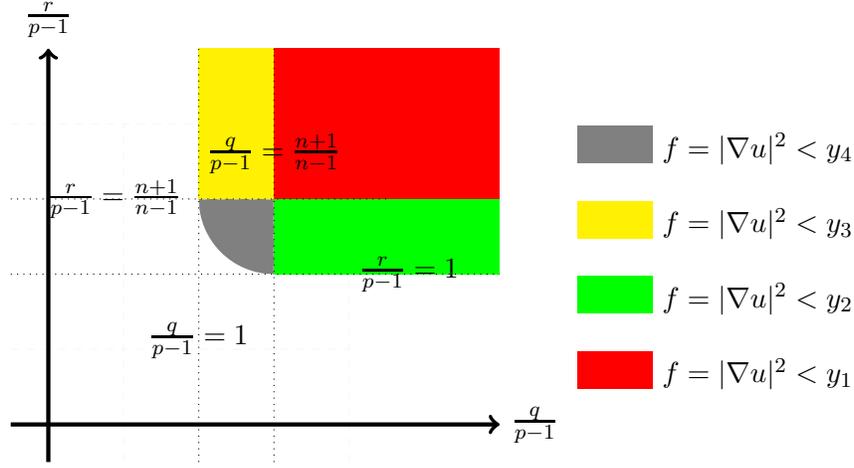
\begin{figure}
	\begin{tikzpicture}
		\fill[gray](3,3) circle (1);
		\draw[help lines, color=red!5, dashed] (-0.5,-0.5) grid (4,4);
		\path[fill=red](6, 3)--(3,3)--(3,5)--(6, 5);
		\path[fill=green](6, 2)--(3,2)--(3,3)--(6, 3);
		\path[fill=yellow](2, 5)--(2,3)--(3,3)--(3, 5);
		\draw[dotted] (2 ,-0.5) --(2, 1.5)  node[anchor=north]{$\frac{q}{p-1}= 1$} -- (2 ,4) -- (2 ,5);
		\draw[dotted] (3,-0.5) --(3, 1.5) -- (3,4)  node[anchor=north]{$\frac{q}{p-1}= \frac{n+1}{n-1}$} -- (3,5);
		\draw[dotted] (-0.5, 2 )--(2, 2)-- (4, 2) node[anchor=west]{$\frac{r}{p-1}=1$}  -- (6, 2);
		\draw[dotted] (-0.5,3) --(1.9,3)  node[anchor=east]{$\frac{r}{p-1}= \frac{n+1}{n-1}$} -- (4,3) -- (4.5,3);
		\draw[->,ultra thick] (-0.5,0)--(6,0) node[right]{$\frac{q}{p-1}$};
		\draw[->,ultra thick] (0,-0.5)--(0,5) node[above]{$\frac{r}{p-1}$};
	\end{tikzpicture}			
	\begin{tikzpicture}
		\path[fill=green!0](2, 2)--(3,2)--(3,2.5)--(2,2.5);
		\path[fill=red](2, 3)--(3,3)--(3,3.2) node[anchor=west]{$f=|\nabla u|^2<y_1$}--(3,3.5)--(2,3.5);				
		\path[fill=green](2, 4)--(3,4)--(3,4.2) node[anchor=west]{$f=|\nabla u|^2<y_2$}--(3,4.5)--(2,4.5);
		\path[fill=yellow](2, 5)--(3,5)--(3,5.2) node[anchor=west]{$f=|\nabla u|^2<y_3$}--(3,5.5)--(2,5.5);
		\path[fill = gray](2, 6)--(3,6)--(3,6.2) node[anchor=west]{$f=|\nabla u|^2<y_4$}--(3,6.5)--(2,6.5);
	\end{tikzpicture}
	\caption{When $b<0$ and $c<0$, the region of $(\frac{q}{p-1}, \frac{r}{p-1})$ in $W_1, W_2, W_3, W_4$.}
	\label{fig6}
\end{figure}
\subsection{Some further comments and open problems}\

Although the arguments presented in the above subsection give the upper bounds of constants $c(n, p, q, r)$ for different cases by obvious
formula, we can not confirm whether these bounds are sharp or not.
	
In the case the coefficients $b$ and $c$ associated with equation \eqref{equ0} satisfy $bc<0$, i.e., $b$ and $c$ are of different signs, it seems that the method employed here is not valid. However, from the results of Theorem E in \cite{MR4150912} we know that some Liouville properties hold true for positive solutions to \eqref{equ0} on $\mathbb{R}^n$ with $n\geq 3$ if $|b|$ is small enough, $c=1$, $p=2$, $1<r<\frac{n+2}{n-2}$ and $q=\frac{2r}{r+1}$ (also see \cite{filippucci2022priori}).

Naturally, one wants to know whether or not the Cheng-Yau's logarithmic gradient estimate or Liouville property still holds true for positive solutions to \eqref{equ0} with $bc<0$ on a complete Riemannian manifold?

In fact, Bidaut-V\'{e}ron, Garcia-Huidobro and V\'{e}ron \cite{Veron2021} 
studied recently the properties of nonnegative functions satisfying (E) $-\Delta u+m|\nabla u|^q-u^r=0$
in a domain of $\mathbb{R}^n$ when $r > 1$, $m > 0$ and $1 < q < r$. 
They concentrate their analysis on the solutions of (E) with an isolated singularity, or in an exterior domain, or in the whole space. 

The existence of such solutions and their behaviours depend strongly on the values of the exponents $r$ and $q$.

Very recently, Bidaut-V\'{e}ron and V\'{e}ron in \cite{Veron2023} also studied the local properties of positive solutions of the equation $\Delta u - m|\nabla u|^q + u^r =0$ in a punctured domain $\Omega\setminus\{0\}$ of $\mathbb{R}^n$ or in an exterior domain $\mathbb{R}^n\setminus B_R(0)$ in the range $\min\{r, q\} >1$ and $m > 0$. They proved a series of a priori estimates depending $r$ and $q$.

But, it seems that the methods adopted in \cite{Veron2021, Veron2023} are not valid for the case $(M,g)$ is a Riemannian manifold.
\medskip
	
\noindent {\it\bf{Acknowledgements}}: The author Y. Wang is supported by National key Research and Development projects of China (Grant No. 2020YFA0712500).
\medskip
	
\bibliographystyle{plain}

\begin{thebibliography}{10}
	
\bibitem{MR1004713}
Marie-Fran\c{c}oise Bidaut-V\'{e}ron.
\newblock Local and global behavior of solutions of quasilinear equations of {E}mden-{F}owler type.
\newblock {\em Arch. Rational Mech. Anal.}, 107(4): 293--324, 1989.

\bibitem{MR1134481} Marie-Fran\c{c}oise Bidaut-V\'{e}ron and Laurent V\'{e}ron.
\newblock Nonlinear elliptic equations on compact {R}iemannian manifolds and asymptotics of {E}mden equations.
\newblock {\em Invent. Math.}, 106(3): 489--539, 1991.
	
\bibitem{MR3261111}
Marie-Fran\c{c}oise Bidaut-V\'{e}ron, Marta Garcia-Huidobro, and Laurent V\'{e}ron.
\newblock Local and global properties of solutions of quasilinear {H}amilton-{J}acobi equations.
\newblock {\em J. Funct. Anal.}, 267(9): 3294--3331, 2014.
	
\bibitem{MR4150912}Marie-Fran\c{c}oise Bidaut-V\'{e}ron, Marta Garcia-Huidobro, and Laurent V\'{e}ron.
\newblock A priori estimates for elliptic equations with reaction terms involving the function and its gradient.
\newblock {\em Math. Ann.}, 378(1-2): 13--56, 2020.

\bibitem{MR4043662} Marie-Fran\c{c}oise Bidaut-V\'{e}ron, Marta Garcia-Huidobro, and Laurent V\'{e}ron.
\newblock Radial solutions of scaling invariant nonlinear elliptic equations with mixed reaction terms.
\newblock {\em Discrete Contin. Dyn. Syst.}, 40(2): 933--982, 2020.

\bibitem{Veron2021} Marie-Fran\c{c}oise Bidaut-V\'{e}ron, Marta Garcia-Huidobro, and Laurent V\'{e}ron.	
\newblock Singular solutions of some elliptic equations involving mixed absorption-reaction
\newblock{\em Discrete Contin. Dyn. Syst.} 42 (2022), no. 8, 3861--3930.

\bibitem{Veron2023} Marie-Fran\c{c}oise Bidaut-V\'{e}ron and Laurent V\'{e}ron.
\newblock Local behaviour of the solutions of the Chipot–Weissler equation.
\newblock{\em Calc. Var. PDE}, (2023) 62:241.
	
\bibitem{MR982351} Luis~A. Caffarelli, Basilis Gidas, and Joel Spruck.
\newblock Asymptotic symmetry and local behavior of semilinear elliptic equations with critical {S}obolev growth.
\newblock {\em Comm. Pure Appl. Math.}, 42(3): 271--297, 1989.
	
\bibitem{MR1121147} Wen-Xiong Chen and Congming Li.
\newblock Classification of solutions of some nonlinear elliptic equations.
\newblock{\em Duke Math. J.}, 63(3): 615--622, 1991.
	
\bibitem{MR385749}Shiu-Yuen Cheng and Shing-Tung Yau.
\newblock Differential equations on {R}iemannian manifolds and their geometric applications.
\newblock {\em Comm. Pure Appl. Math.}, 28(3): 333--354, 1975.

\bibitem{CM} T. H. Colding and W. P. Minicozzi II.
\newblock Liouville properties.
\newblock {\em ICCM Not. 7 }(2019), no. 1, 16-26.

\bibitem{MR1491451} Tobias~H. Colding and William~P. Minicozzi, II.
\newblock Harmonic functions on manifolds.
\newblock {\em Ann. of Math. (2)}, 146(3): 725--747, 1997.
	
\bibitem{MR1472893} Tobias~H. Colding and William~P. Minicozzi, II.
\newblock Harmonic functions with polynomial growth.
\newblock {\em J. Differential Geom.}, 46(1): 1--77, 1997.

\bibitem{Dancer} Edward Norman Dancer.
\newblock Superlinear problems on domains with holes of asymptotic shape and exterior problems.
\newblock{Math. Z.} 229 (1998), no. 3, 475–491.

\bibitem{MR0709038}Emmanuele DiBenedetto.
\newblock {$C\sp{1+\alpha }$} local regularity of weak solutions of degenerate elliptic equations.
\newblock {\em Nonlinear Anal.}, 7(8): 827--850, 1983.
	
\bibitem{filippucci2022priori} Roberta Filippucci, Yuhua Sun, and Yadong Zheng.
\newblock A priori estimates and liouville type results for quasilinear elliptic equations involving gradient terms.
\newblock {\em arXiv:2205.07484}, 2022.
	
\bibitem{MR615628} Basilis Gidas and Joel Spruck.
\newblock Global and local behavior of positive solutions of nonlinear elliptic equations.
\newblock {\em Comm. Pure Appl. Math.}, 34(4): 525--598, 1981.
	
\bibitem{han2023gradient}
Dong Han, Jie He, and Youde Wang.
\newblock Gradient estimates for ${\Delta}_pu-|\nabla u|^q+b(x)|u|^{r-1}u=0$ on a complete Riemannian manifold and liouville type theorems.
\newblock {\em arXiv:2309.03510}, pages 1--37, 2023.
	
\bibitem{he2023nashmoser}
Jie He, Jingchen Hu, and Youde Wang.
\newblock Nash-moser iteration approach to gradient estimate and liouville property of quasilinear elliptic equations on complete Riemannian manifolds.
\newblock {\em arXiv:2311.02568}, 2023.
	
\bibitem{he2023gradient}
Jie He, Youde Wang, and Guodong Wei.
\newblock Gradient estimate for solutions of the equation $\Delta_pv +av^q=0$
on a complete riemannian manifold.
\newblock {\em Math. Z.}, 306(3): Paper No. 46, 19, 2024.
	
\bibitem{MR4594369}
Guangyue Huang and Liang Zhao.
\newblock Liouville type theorems for nonlinear {$p$}-{L}aplacian equation on complete noncompact {R}iemannian manifolds.
\newblock {\em Chinese Ann. Math. Ser. B}, 44(3):379--390, 2023.
	
\bibitem{MR2518892}
Brett Kotschwar and Lei Ni.
\newblock Local gradient estimates of {$p$}-harmonic functions, {$1/H$}-flow, and an entropy formula.
\newblock {\em Ann. Sci. \'{E}c. Norm. Sup\'{e}r. (4)}, 42(1):1--36, 2009.

\bibitem{Li-Yau} Peter Li and Shing-Tung Yau
\newblock On the parabolic kernel of the Schr¨odinger operator.
\newblock {\em Acta Math.} 156 (1986), no. 3-4, 153-201.

\bibitem{Li-Z} YanYan Li and Lei Zhang.
\newblock Liouville-type theorems and Harnack-type inequalities for semilinear elliptic
equations.
\newblock{J. Anal. Math.} 90 (2003), 27-87.
	
\bibitem{MR833413}
Pierre-Louis Lions.
\newblock Quelques remarques sur les probl\`emes elliptiques quasilin\'{e}aires du second ordre.
\newblock {\em J. Analyse Math.}, 45:234--254, 1985.
	
\bibitem{MR1879326}
Enzo Mitidieri and Stanislav~I Pokhozhaev.
\newblock A priori estimates and the absence of solutions of nonlinear partial differential equations and inequalities.
\newblock {\em Tr. Mat. Inst. Steklova}, 234:1--384, 2001.
	
\bibitem{MR829846}Wei-Ming Ni and James Serrin.
\newblock Nonexistence theorems for singular solutions of quasilinear partial differential equations.
\newblock {\em Comm. Pure Appl. Math.}, 39(3):379--399, 1986.
	
\bibitem{MR2350853} Peter Pol\'{a}\v{c}ik, Pavol Quittner, and Philippe Souplet.
\newblock Singularity and decay estimates in superlinear problems via {L}iouville-type theorems. {I}. {E}lliptic equations and systems.
\newblock {\em Duke Math. J.}, 139(3):555--579, 2007.
	
\bibitem{saloff1992uniformly} Laurent Saloff-Coste.
\newblock Uniformly elliptic operators on Riemannian manifolds.
\newblock {\em Journal of Differential Geometry}, 36(2): 417--450, 1992.
	
\bibitem{MR1946918} James Serrin and Henghui Zou.
\newblock Cauchy-{L}iouville and universal boundedness theorems for quasilinear elliptic equations and inequalities.
\newblock {\em Acta Math.}, 189(1): 79--142, 2002.

\bibitem{S-2} Philippe Souplet
\newblock Universal estimates and Liouville theorems for superlinear problems without scale invariance.
\newblock {\em Discrete Contin. Dyn. Syst.} 43(2023), no. 3-4, 1702–1734.

	
\bibitem{MR2522424}Philippe Souplet.
\newblock The proof of the {L}ane-{E}mden conjecture in four space dimensions.
\newblock {\em Adv. Math.}, 221(5): 1409--1427, 2009.
	
\bibitem{MR3275651}
Chiung-Jue~Anna Sung and Jiaping Wang.
\newblock Sharp gradient estimate and spectral rigidity for {$p$}-{L}aplacian.
\newblock {\em Math. Res. Lett.}, 21(4): 885--904, 2014.
	
\bibitem{MR0727034} Peter Tolksdorf.
\newblock Regularity for a more general class of quasilinear elliptic equations.
\newblock {\em J. Differential Equations}, 51(1): 126--150, 1984.
	
\bibitem{MR0474389} Karen Uhlenbeck.
\newblock Regularity for a class of non-linear elliptic systems.
\newblock {\em Acta Math.}, 138(3-4): 219--240, 1977.

\bibitem{wang2021gradient1}
Jie Wang and You-De Wang.
\newblock Gradient estimates for ${\Delta} u + a(x)u\log u + b(x)u=0$ and its parabolic counterpart under integral {R}icci curvature bounds.
\newblock {\em arXiv:2109.05235}: 1--37, 2021 (to appear in CAG).
	
\bibitem{MR2880214}Xiaodong Wang and Lei Zhang.
\newblock Local gradient estimate for {$p$}-harmonic functions on {R}iemannian manifolds.
\newblock {\em Comm. Anal. Geom.}, 19(4): 759--771, 2011.

\bibitem{Wang}You-De Wang.
\newblock Harmonic maps from noncompact Riemannian manifolds with non-negative Ricci curvature outside a compact set.
\newblock {\em Proc. Roy. Soc. Edinburgh Sect. A}, 124 (1994), no. 6, 1259-1275.
	
\bibitem{MR4559367}Youde Wang and Guodong Wei.
\newblock On the nonexistence of positive solution to {$\Delta u + au^{p+1} = 0$} on {R}iemannian manifolds.
\newblock {\em J. Differential Equations}, 362: 74--87, 2023.
	
\bibitem{MR431040} Shing-Tung Yau.
\newblock Harmonic functions on complete {R}iemannian manifolds.
\newblock {\em Comm. Pure Appl. Math.}, 28: 201--228, 1975.
	
\bibitem{MR3912761}Liang Zhao.
\newblock Liouville theorem for weighted {$p$}-{L}ichnerowicz equation on smooth metric measure space.
\newblock {\em J. Differential Equations}, 266(9): 5615--5624, 2019.
\end{thebibliography}

\end{document}